\documentclass[12pt,a4paper,twoside]{amsart}
\usepackage[a4paper,left=30mm,right=30mm,top=30mm,bottom=30mm]{geometry}
\usepackage{setspace}
\usepackage{graphicx} 
\usepackage{amsmath, mathtools, amssymb, amsthm}
\usepackage[symbol]{footmisc}
\usepackage{xcolor}
\usepackage{subcaption}
\usepackage{tikz}
\usepackage{hyperref}
\hypersetup {colorlinks=true, linkcolor=black, urlcolor=blue, citecolor=blue}
\usepackage{dsfont}
\usepackage{stmaryrd}
\usepackage{float}
\usepackage[rightcaption]{sidecap}

\sidecaptionvpos{figure}{t}

\newtheorem{theorem}{Theorem}[section]

\newtheorem{lemma}[theorem]{Lemma}
\newtheorem{proposition}[theorem]{Proposition}
\newtheorem{remark}[theorem]{Remark}
\newtheorem{definition}[theorem]{Definition}

\numberwithin{equation}{section}

\newcommand{\N}{\mathbb{N}}
\newcommand{\Z}{\mathbb{Z}}
\newcommand{\R}{\mathbb{R}}
\newcommand{\C}{\mathbb{C}}

\newcommand{\rot}{\operatorname{curl}}
\newcommand{\curl}{\rot}
\renewcommand{\div}{\operatorname{div}}
\newcommand{\dd}{\, \mathrm{d}}
\newcommand{\dx}{\, \mathrm{d}x}
\newcommand{\one}[1]{{\bf{1}}_{#1}}
\newcommand{\eps}{\varepsilon}

\newcommand{\del}{\partial}
\newcommand{\Oe}{\Omega_\eta}
\newcommand{\Ee}{E^\eta}
\newcommand{\He}{H^\eta}
\newcommand{\weakto}{\rightharpoonup}
\newcommand{\id}{\mathrm{id}}

\newcommand{\loc}{\mathrm{loc}}

\newcommand{\fhi}{\varphi}
\newcommand{\Hd}{\overset{\tikz\draw[fill=black] (0,0) circle (.2ex);}{H}_\#}
\newcommand{\urs}{u^{\psi}_r} 
\newcommand{\vrs}{v^{\psi}_r} 
\renewcommand{\hom}{\text{hom}}

\begin{document}
%\doublespacing   
\bibliographystyle{abbrv}

\pagestyle{myheadings} \markboth{Polarization
 filter}{B. Schweizer and D.~Wiedemann}

\thispagestyle{empty}
\begin{center}
  ~\vskip4mm {\Large\bf
    Interface conditions for Maxwell's equations \\[3mm]
    by homogenization of thin inclusions:\\[3mm]
    transmission, reflection or polarization
  }\\[7mm]
  {\large B.~Schweizer\footnotemark[1] and D.~Wiedemann\footnotemark[1]}\\[3mm]
 
 January 29, 2025 \\[5mm]
\end{center}

\footnotetext[1]{Technische Universität Dortmund, Fakult\"at f\"ur
  Mathematik, Vogelspothsweg 87, D-44227
  Dortmund. Ben.Schweizer$@$tu-dortmund.de,
  David.Wiedemann$@$tu-dortmund.de}

\begin{center}
 \vskip3mm
 \begin{minipage}[c]{0.87\textwidth}
   {\bf Abstract:} We consider the time-harmonic Maxwell equations in
   a complex geometry. We are interested in geometries that model
   polarization filters or Faraday cages. We study the situation that
   the underlying domain contains perfectly conducting inclusions, the
   inclusions are distributed in a periodic fashion along a
   surface. The periodicity is $\eta>0$ and the typical scale of the
   inclusion is $\eta$, but we allow also the presence of even smaller
   scales, e.g.~when thin wires are analyzed. We are interested in the
   limit $\eta\to 0$ and in effective equations. Depending on
   geometric properties of the inclusions, the effective system can
   imply perfect transmission, perfect reflection or polarization.
 
 %35Q61: Maxwell equations
 %78A25: Electromagnetic theory (general)
 %78M35: Asymptotic analysis in optics and electromagnetic theory
 %78M40: Homogenization in optics and electromagnetic theory
 %35B40: Asymptotic behavior of solutions to PDEs
 \end{minipage}\\[7mm]
\end{center}

\section{Introduction}
The propagation of electromagnetic waves can be influenced by thin
metallic structures. Probably best known is the fact that
(appropriately designed) meshes of metallic wires are impenetrable for
the waves of a micro-wave oven and they can protect the user. Vice
versa, metallic surfaces can be made penetrable for electromagnetic
waves by means of thin slits. Another effect can be the polarization
of the wave. We are interested in a mathematical analysis of such
effects.

We study the time-harmonic Maxwell equations in a domain that contains
a thin layer of a meta-material: A perfectly conductive material is
distributed in a periodic fashion along a surface. The material can be
distributed in a multi-scale fashion, e.g.~a wire structure with a
periodicity $\eta>0$ and a thickness of the wires of a size much
smaller than $\eta$. Another example would be a plate structure with
small slits in thin plates. We are interested in the limit $\eta\to 0$
and in an effective description of this situation. Our results show
that the geometrical and topological properties of the inclusions are
crucial for the qualitative effect of the meta-material. The effective
system can show reflection or polarization. For structures with slits
or structures that do not have certain connectivity properties, the
meta-material can also be invisible in the sense that it does not
appear in the effective equations.

\smallskip Maxwell's equations describe electromagnetic waves with an
electric field $E$ and a magnetic field $H$. The equations use two
coefficients, the permeability $\mu$ and the permittivity $\eps$.
Since we are interested in complex geometries, the domain of interest
$\Oe \subset \Omega\subset\R^3$ depends on a (small) parameter
$\eta>0$.  In the time-harmonic setting with a fixed frequency
$\omega>0$, the system reads
\begin{subequations}\label{eq:Strong:Maxwell:eta}
  \begin{align}\label{eq:Strong:Maxwell:eta:1}
    \rot\Ee &= i \omega \mu \He + f_h &&\textrm{ in } \Omega_\eta \, , 
    \\\label{eq:Strong:Maxwell:eta:2}
    \rot\He &= - i \omega \varepsilon \Ee +f_e &&\textrm{ in } \Omega_\eta \, ,
    \\\label{eq:Strong:Maxwell:eta:3}
    \Ee \times \nu &= 0 &&\textrm{ on } \partial \Omega_\eta \, .
  \end{align}
\end{subequations}
The last equation is a boundary condition that models a perfectly
conducting material outside $\Omega_\eta$, it uses the exterior normal
$\nu$ of $\Omega_\eta$. The index $\eta$ indicates that we will
analyze a fine-scale structure. The solution is a map
$(\Ee,\He) \colon \Omega_\eta\ni x\mapsto (\Ee,\He)(x) \in \C^3\times
\C^3$. Since the solution depends on the domain $\Omega_\eta$, it
depends on $\eta$. The right hand side $f_e = f_e(x)$ of the second
equation models prescribed external currents; we include
$f_h = f_h(x)$ to treat a more general and more symmetric system.

\smallskip 

We consider the macroscopic domain $\Omega$, which is separated by
$\Gamma$ into an upper part $\Omega_+$ and a lower part $\Omega_-$ as
follows:
\begin{align}\label{eq:def:Omega-Gamma}
  \Omega &\coloneqq (0,1)^2 \times (-1,1) \,, & \Gamma
  &\coloneqq (0,1)^2 \times \{0\}\,,\\\label{eq:def:Omega-pm}
  \Omega_+ &\coloneqq (0,1)^2 \times (0,1) \,,&\Omega_-
  &\coloneqq (0,1)^2 \times (-1,0) \,.
\end{align}
We assume that there exists a sequence of sets
$\Sigma_\eta \subset \Omega$ which represents the perfectly conducting
inclusion. The sets $\Sigma_\eta$ concentrate at the interface
$\Gamma = \{x_3 = 0\}$ in the sense that
$\Sigma_\eta \subset (0,1)^2 \times (0, \eta)$. The domain $\Oe$ is
given by $\Oe \coloneqq \Omega \setminus \Sigma_\eta$.  We are
interested in the effective interface condition along $\Gamma$.

\subsection{Asymptotic connectedness and disconnectedness of
  conductors}

In the simplest case, the micro-structure consists of a periodic array
of parallel wires, let us assume that they are oriented in direction
$e_1$. We will introduce a definition for ``asymptotic connectedness''
and will see that such parallel wires are asymptotically connected in
direction $e_1$ and asymptotically disconnected in direction $e_2$. In
the limit $\eta\to 0$, this micro-structure reflects electromagnetic
waves that have an electric field parallel to $e_1$, when the electric
field is orthogonal to $e_1$, the structure is invisible for the
electromagnetic wave.  Since a general wave can be understood as a
superposition of the above described waves, the microstructure leads
to a polarization of the general electromagnetic wave, compare
Figure~\ref{fig:polarization-wire-structure}.

\vspace{0.4cm}

\begin{SCfigure}[][ht]
  \begin{minipage}{0.7\linewidth}
    \vspace{-1.2cm}
    \includegraphics[width=\linewidth]{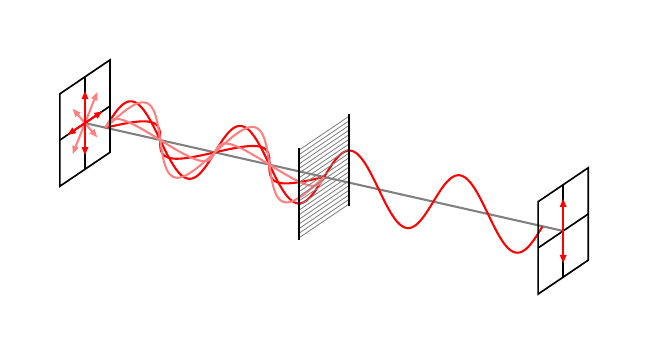}
  \end{minipage}
  \captionsetup{width=0.97\linewidth}
  \caption{Polarization: Electromagnetic waves (represented by the
    electric field) interacting with a mesh of wires. Components that
    are orthogonal to the wires can pass the
    microstructure.\label{fig:polarization-wire-structure}}
\end{SCfigure}

The limit of \eqref{eq:Strong:Maxwell:eta} depends on an asymptotic
analysis of the connectedness properties. We define a notion of
\textit{asymptotic (dis-)connectedness} in certain directions for the
obstacle set by means of the existence of cell functions, see
Definitions \ref {def:Analytical-DisConnectedness} and \ref
{def:Analytical-Connectedness}.  With these notions, the
homogenization of \eqref{eq:Strong:Maxwell:eta} leads to one of the
following three homogenization results: perfect transmission (inactive
interface), perfect reflection or polarization.

Having used cell functions in order to define asymptotic
(dis-)connectedness, it is an important task to study the
connectedness properties for concrete geometrical objects.  Our
investigation is focused on thin parallel wires that may also have
thin slits (gaps in the wire). The existence of the corresponding cell
functions and, thus, the asymptotic connectedness or disconnectedness
depends on the asymptotic behavior of the parameters describing the
structures; we will see that the asymptotics of the number
$\eta |\ln(r_\eta)|$ are decisive, where $r_\eta$ is the radius of the
wire, see Proposition \ref {prop:construction:Phi-wires-thin}.

To enlarge the class of domains under consideration, we introduce a
monotonicity and a deformation argument. Loosely speaking: Enlarging a
connected structure leaves the structure connected, making a
disconnected structure smaller, leaves a disconnected structure
disconnected (large and small is meant here in the sense of set
inclusions, see Proposition \ref {prop:GeometricMonotonicity}).  The
deformation result provides that Lipschitz deformations of a
(dis-)connected structure is again (dis-)connected, see
Proposition~\ref {prop:construction:Psi-wires-perfect:curved}.

\subsection{Interface conditions of the limit system}
To formulate the limit system, we use the following notation: Whenever
a scalar function $u \colon \Omega \to \C$ possesses traces on the
interface $\Gamma$, we write $u|_{\Gamma,+}$ for the trace of
$u|_{\Omega_+}$, i.e.~the boundary values taken from above the
interface. Correspondingly, we write $u|_{\Gamma,-}$ for the trace of
$u|_{\Omega_-}$, which are the boundary values from below. For the
jump of such a function across $\Gamma$ we write
$\llbracket u \rrbracket_\Gamma \coloneqq u|_{\Gamma,+} -
u|_{\Gamma,-}$.  We will actually use the brackets only to formulate
that a function $u$ has no jump across $\Gamma$ by writing
$\llbracket u \rrbracket_\Gamma = 0$.

In the case $\llbracket u \rrbracket_\Gamma = 0$, the function $u$ has
no jump across $\Gamma$ and we can write
$u|_{\Gamma} \coloneqq u|_{\Gamma,+} = u|_{\Gamma,-}$. In fact, we
will use traces only in this case.  For vector valued functions
$E \colon \Omega \to \C^3$ of the class
$H(\curl, \Omega\setminus \Gamma)$, there is no continuous trace
operator for $E$, but there are continuous trace operators for the
tangential components.

\smallskip {\em Limit system.}  In our setting, the homogenization of
\eqref{eq:Strong:Maxwell:eta} results in the following limit
system for $E^\hom$, $H^\hom$:
\begin{subequations}
  \label{eq:Maxwell-strong-E-H:Limit}
  \begin{align}\label{eq:Maxwell-strong-E-H:Limit:1}
    \rot E^\hom &= i \omega \mu H^\hom + f_h
    &&\textrm{ in } \Omega \, , 
    \\\label{eq:Maxwell-strong-E-H:Limit:2}
    \rot H^\hom &= - i \omega \varepsilon E^\hom +f_e
    &&\textrm{ in } \Omega \setminus \Gamma\,,
    \\
    E^\hom \times \nu &= 0
    && \textrm{ on } \partial \Omega\,.
       \label{eq:Maxwell-strong-E-H:Limit:3}
  \end{align}
\end{subequations}
The system \eqref{eq:Maxwell-strong-E-H:Limit} must be complemented by
interface conditions at $\Gamma$ (this is a consequence of the fact
that \eqref{eq:Maxwell-strong-E-H:Limit:2} is not imposed on all of
$\Omega$). Depending on the asymptotic connectedness property of
$\Sigma_\eta$, one of the interface conditions \eqref
{eq:Maxwell-strong-E-H:Reflecting} or \eqref
{eq:Maxwell-strong-E-H:Inactive} or \eqref
{eq:Maxwell-strong-E-H:Polarizing} completes the system.

We note that \eqref{eq:Maxwell-strong-E-H:Limit:1} implies that the
tangential components of $E^\hom$ have no jump across $\Gamma$:
\begin{equation*}
  \llbracket E^\hom_1 \rrbracket _\Gamma=0 \,,
  \qquad\quad
  \llbracket E^\hom_2 \rrbracket _\Gamma = 0 \,,
\end{equation*}
hence these conditions are always satisfied.  By contrast, the
tangential components of $H^\hom$ can have jumps across $\Gamma$ since
\eqref{eq:Maxwell-strong-E-H:Limit:2} imposes an equation only on
$\Omega\setminus \Gamma$.

\smallskip
\textbf{Reflecting interface:} If the obstacles are asymptotically
connecting in both directions $e_1$ and $e_2$, we obtain the interface
conditions
\begin{align}\label{eq:Maxwell-strong-E-H:Reflecting}
  E^\hom_1|_\Gamma =0 \,,\qquad\quad E^\hom_2|_\Gamma= 0 \,.
\end{align}
{\em Remark:} The interface condition \eqref
{eq:Maxwell-strong-E-H:Reflecting} means that the domains $\Omega_+$
and $\Omega_-$ are separated. The solution $(E^\hom, H^\hom)$ can
equivalently be found by solving the equations in $\Omega_+$ and
$\Omega_-$ with boundary conditions $E^\hom \times\nu = 0$ on
$\partial \Omega_-$ and $\partial \Omega_+$.

\smallskip
\textbf{Inactive interface:} If the obstacles are asymptotically
disconnected in both directions $e_1$ and $e_2$, we obtain the
interface conditions
\begin{align}
  \label{eq:Maxwell-strong-E-H:Inactive}
  \llbracket H^\hom_1 \rrbracket_\Gamma =0 \,,
  \qquad\quad	\llbracket H^\hom_2 \rrbracket_\Gamma = 0\,.
\end{align}
{\em Remark:} 
The relations \eqref{eq:Maxwell-strong-E-H:Inactive} are satisfied if
and only if \eqref{eq:Maxwell-strong-E-H:Limit:2} holds in the entire
domain $\Omega$. We may therefore also say in this case: The
micro-structured interface has no effect in the limit system.

\smallskip \textbf{Polarizing interface:} Let $i \in \{1,2\}$ and
$j \coloneqq 3-i$ be fixed.  If the obstacles are asymptotically
connecting for the direction $e_i$ and asymptotically disconnected for
the direction $e_j$, we obtain the interface conditions
\begin{align}\label{eq:Maxwell-strong-E-H:Polarizing}
  E^\hom_i|_\Gamma = 0 \,,
  \qquad\quad \llbracket H^\hom_i \rrbracket_\Gamma = 0 \,.
\end{align}

\subsection{Outlook regarding highly conducting obstacles}
System \eqref{eq:Strong:Maxwell:eta} models perfect conductors in the
inclusion $\Sigma_\eta$. Another interesting model is constructed by
introducing a high conductivity in the inclusion.  For a sequence
$\gamma_\eta \to \infty$ and $\eps_0 \in \R$, one may consider
$\eps_\eta = \eps_0 + i \gamma_\eta\, \one{\Sigma_\eta}$ as
permittivity in Maxwell's equations. The homogenization limit can be
considered also for this model. In this slightly more complex setting,
one has to take into account the asymptotics of $\gamma_\eta$ and its
relation to the scaling and the geometry of $\Sigma_\eta$.  We refrain
from an analysis of this interesting setting in the present
contribution. Highly conducting media were investigated in related
context in \cite{MR2576911, MR2262964, Bouchitte-Schweizer-Max-2010,
  Lamacz-Schweizer-Neg}.

\subsection{Literature} In the original form, Maxwell's equations form
a time-dependent system, but one is very often interested in the
analysis for a fixed frequency $\omega>0$.  The ansatz
$u(x,t) = u(x) e^{-i\omega t}$ leads from Maxwell's equations to the
time-harmonic Maxwell system \eqref{eq:Strong:Maxwell:eta}.  We note
that the same procedure leads from the scalar wave equation to the
Helmholtz equation.  Both systems, Maxwell and Helmholtz, describe
wave phenomena and their analysis has many similarities. Accordingly,
we describe here also some work on the Helmholtz equation. We focus on
works that are related to homogenization and suppress many interesting
fields such as radiation conditions, periodic media or domain
truncation methods.  Regarding general mathematical concepts for
Maxwell's equations, we refer to \cite{Kirsch-Buch-2015} and
\cite{Monk2003a}.

\smallskip {\em Homogenization in the bulk.} The theory of
homogenization started with the analysis of periodic systems -- a
periodicity in all directions. It can be either the domain that
contains, e.g.~small and periodically spaced inclusions, or it can be
a coefficient in the equation that is periodic with a small
periodicity, say $\eta>0$. The homogenization of such system can be
performed, e.g.~with the method of two-scale convergence
\cite{MR1185639}. A homogenization of the Maxwell system in this
classical spirit is performed in \cite{MR2029130}, error estimates
have been obtained in \cite {MR3995044}, we mention \cite {MR3810505,
  MR3578028} for numerical aspects. The analysis becomes more
intricate when the system is near resonance, we refer to \cite
{Schw-resonance-2017} for an overview. Wires with large absolute
values of the permittivity $\eps$ can lead to resonances, this was
investigated in \cite {MR2262964} for an essentially two-dimensional
setting, in \cite {MR2576911} for a setting towards three space
dimensions, in \cite {MR3348421} for a random permittivity in the
wires.

Other resonances are possible for three-dimensional inclusions in the
bulk, namely the famous split-ring resonators. They were first
analyzed mathematically in \cite {Bouchitte-Schweizer-Max-2010}, the
homogenization was performed, the homogenized system can show an
effective negative permeability $\mu$.  The topology is important in
this resonance: For positive $\eta$, the rings are simply connected,
but they are closing in the limit, which changes their topology. In
particular, one has to analyze a microstructure that has some
geometrical features below the $\eta$-scale.  The results were
transferred to perfectly conducting inclusions in \cite
{Lipton-Schweizer-2018}. The generation of a negative index material
occurs when macroscopically connecting wires are additionally
included, see \cite{Lamacz-Schweizer-Neg}. A more general
investigation of topological implications was performed in \cite
{Schweizer-Urban-2018} for perfect conductors and in \cite
{Ohlberger-Schweizer-Urban-Verfuehrt-2020} for high-contrast media.
One result is that if the inclusion is connected in one direction, the
limiting $E$ field must vanish in that component.

A composite medium for Maxwell's equations in which one component has
a negative index was homogenized in \cite {MR4401386}. A high-contrast
homogenization for a large electric permittivity in a periodic
distribution was performed in \cite {MR3343064}.  A non-trivial
asymptotic geometry was analyzed in two dimensions for the Helmholtz
equation in \cite {MR3667564, Schweizer-low-freq}.

We note that a special scaling of a micro-structure was used in \cite
{Bouchitte-Schweizer-Plasmonic-2013}, again for the Helmholtz
equation: The interface has thickness of order $1$, the
micro-structure is $\eta$-periodic in the tangential direction, the
aim is to understand the coupling between upper and lower surface of
this interface; the result is the possible effect of perfect
transmission due to a resonance effect.  Sound absorption in this
scaling was investigated in \cite {Donato-Lamacz-Schweizer-2022}.

\smallskip {\em Homogenization along a surface.}  The above list of
references indicates that very interesting effects, most of them
related to resonance, can occur when small structures are distributed
in the bulk (and, often, three scales are used in such
constructions). The geometry of the present article is such that,
along a manifold (here: a two-dimensional surface in three dimensions,
compare \eqref {eq:macro-domain}), a micro-structure is
distributed. The fundamental task is to determine the effective
description of this microstructure.

For the Helmholtz equation with fixed frequency, the result at leading
order is the same as for the Poisson problem. When inclusions of scale
$\eta$ are distributed with typical distances $\eta$, the effective
system is not very interesting: When a Dirichlet condition is imposed
on the boundaries of the inclusions, the effective system contains a
Dirichlet condition at the limiting interface.  When a Neumann
condition is imposed along the inclusion, the limiting interface does
not enter the effective equation. We note that smaller inclusions can
have a different effective behavior, see \cite {MR1493040}.

The results become interesting when one asks for the effects in the
next order of the expansion.  In the fundamental contribution \cite
{DelourmeHaddarJoly-2012}, this was answered: The derivatives of the
leading order solutions must be evaluated at the limiting interface
and they are responsible for a correction term of first order in the
small parameter. We mention that the results of \cite
{DelourmeHaddarJoly-2012} were slightly generalized and simplified in
\cite {Schweizer-Neumann-2020}.

The present contribution can be seen in this context. We do not ask
here for first order correction terms. On the other hand, we treat the
three-dimensional situation and Maxwell's equations; we can therefore
find interesting results even at leading order. Most notable, for
wires: Cancellation effect in one component, invisibility of the
effective surface in the other component. This can be seen like the
effect of a Dirichlet condition in one component, and the effect of a
Neumann condition in the other component.

\subsection{Content layout}
The manuscript is organized as follows. In
Section~\ref{sec:main-results}, we present the geometries under
consideration, formulate the weak form for the $\eta$-problem and the
weak forms for the homogenized systems.  We formulate the main
homogenization result in Theorem \ref {thm:main:homogenization}.  In
Section \ref{sec:Test-functions}, we introduce the notion of
asymptotic (dis-)connectedness by means of the existence of cell
functions. Using these cell functions, we homogenize
\eqref{eq:Strong:Maxwell:eta} for generic geometries that are
asymptotically (dis-)connected in Section \ref{sec:Hom-Maxwell}.
Section \ref{sec:connected-general} illustrates that one can conclude
connectedness properties of one inclusion from the connectedness
properties of another inclusion: Deformations of the inclusion do not
change the connectedness properties. Furthermore: When an inclusion is
connected, every larger inclusion (in the sense of sets) is also
connected.  In Section \ref {sec:wires}, we investigate cylindrical
geometries and show that wires are connecting in one direction and are
disconnected in the other direction. The fact that the correct
critical asymptotics for the radius are indeed found is shown in
Section \ref {sec:criticality}.

\section{Main results}\label{sec:main-results}
In this section, we describe the microscopic and the macroscopic
geometry. We present the weak form of the $\eta$-problem and of the
limit problem with different boundary conditions. Theorem \ref
{thm:main:homogenization} states the main result.

\subsection{Geometry and notations}\label{sec:Geometry}
For a measurable set $U \subset \R^n$, we denote the indicator
function by $\one{U}$, i.e.~${\one{U}(x) = 1}$ for $x \in U$ and
$\one{U}(x) = 0$ for $x \not \in U$. In estimates, we use generic
constants $C$ that may change from one inequality to the next.

We recall the macroscopic geometry of
\eqref{eq:def:Omega-Gamma}--\eqref{eq:def:Omega-pm}: The domain is
$\Omega = (0,1)^2 \times (-1,1)$, the interface
$\Gamma = (0,1)^2 \times \{0\}$ separates
$\Omega_+ = (0,1)^2 \times (0,1)$ and
$\Omega_- = (0,1)^2 \times (-1,0)$.

The small parameter is $\eta >0$. We denote the three dimensional unit
cell by $Y \coloneqq (0,1)^3$. The microscopic geometry is prescribed
by a sequence of compact obstacle shapes
$\Sigma_Y^\eta \subset [0,1]^2\times (0,1) \subset \overline{Y}$. To
extend the structure periodically, we use indices
$k= (k_1, k_2) \in \Z^2$ and set
\begin{equation*}
  \Sigma^\eta_\# \coloneqq \bigcup
  \limits_{k \in \Z^2} (k_1 ,k_2, 0) + \Sigma_Y^\eta \subset \R^3\,.
\end{equation*}
We assume that $\Sigma^\eta_\#$ has a Lipschitz boundary. A scaling of
$\Sigma^\eta_\#$ with $\eta$ provides the inclusion, $\Omega_\eta$ is
the perforated macroscopic domain:
\begin{equation}
 \label{eq:macro-domain}
 \Sigma_\eta \coloneqq
 \eta\, \Sigma^\eta_\# \cap \Omega
 = \bigcup\limits_{k \in \Z^2}
 \eta[(k_1,k_2,0) + \Sigma_Y^\eta] \cap \Omega\,,
 \qquad 
 \Omega_\eta \coloneqq \Omega \setminus \Sigma_\eta \,.
\end{equation}
For an illustration of a possible cell-geometry see Figure
\ref{fig:Y-Sigma}, the resulting domain $\Omega_\eta$ is shown in
Figure~\ref{fig:Omega_Eta}.

\captionsetup[subfigure]{aboveskip=-8pt, belowskip=8pt}
\begin{figure}[ht]
  \centering
  \begin{subfigure}[t]{.48\textwidth}
    \centering
    \raisebox{0.4cm}{\includegraphics[width=\linewidth]{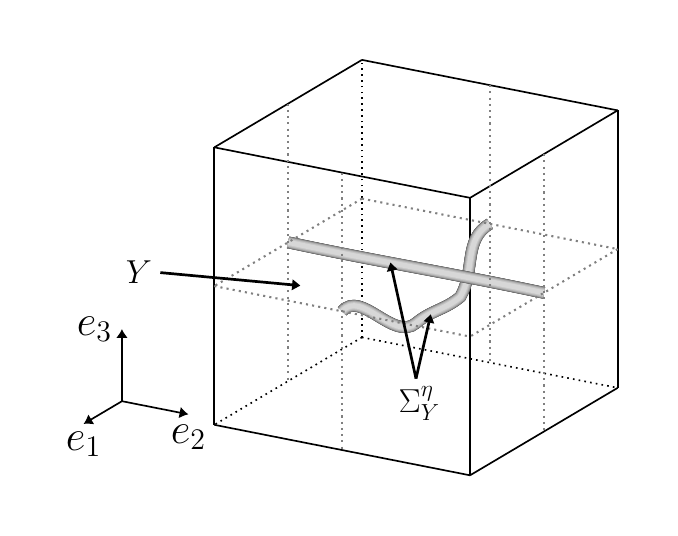}
    } \subcaption{The set $\Sigma_Y^\eta$ representing the conductor
      in the reference cell}
    \label{fig:Y-Sigma}
  \end{subfigure}
  \hspace{\fill}
  \begin{subfigure}[t]{.48\textwidth}
    \centering
    \includegraphics[width=\linewidth]{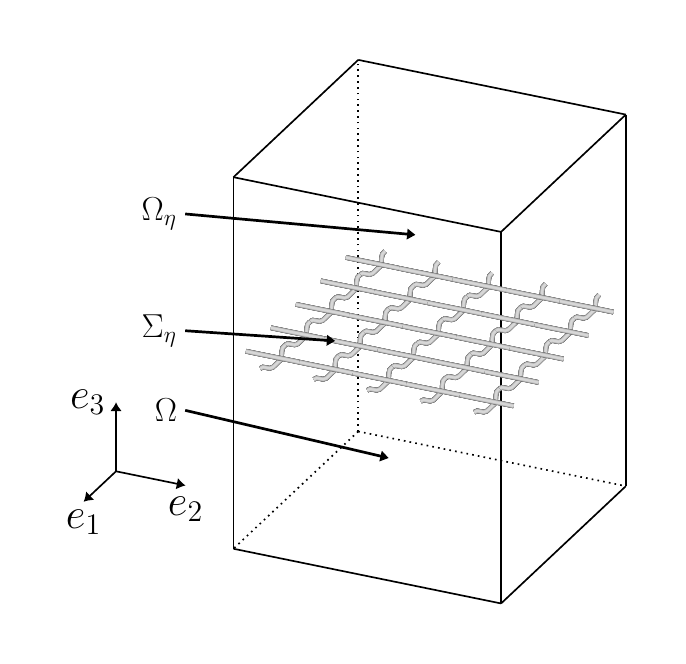}
    \subcaption{The domain $\Omega_\eta$ with obstacle set
      $\Sigma_\eta$ } \label{fig:Omega_Eta}
  \end{subfigure}
  \captionsetup{aboveskip=0pt}
  \caption{Geometry of the conductor}
  \label{fig:geometries}
\end{figure}

\subsection{Weak form of the \texorpdfstring{$\eta$}{eta}-problem}
For the analysis of Maxwell's equations we use the following function
spaces: For a domain $U\subset \R^3$, understanding the curl of an
$L^2(U)$-function in the sense of distributions, we set
\begin{align*}
  H(\rot, U)
  &\coloneqq \{ E \in L^2(U,\C^3) \mid \rot E \in L^2(U,\C^3)\} \,,\\
  H_0(\curl,U)
  &\coloneqq \left\{E\in H(\curl,U)\middle|\int_U 
    E\cdot\curl\psi
    = \int_U \curl E\cdot \psi\ \forall
    \psi\in H(\curl,U)\right\} \,,\\
  H_\loc(\rot, U)
  &\coloneqq \{ E \in L^2_\loc(U,\C^3) \mid \rot E \in L^2_\loc(U,\C^3)\}\,.
\end{align*}
For the first two spaces, we use the norm and the scalar product that
is induced by
\begin{align*}
  \|E \|^2_{H(\rot, U)} \coloneqq
  \int\limits_U \left\{|E|^2 + |\rot E|^2\right\} \, .
\end{align*}
We use the following weak formulation of the $\eta$-problem
\eqref{eq:Strong:Maxwell:eta}: Find
$(\Ee, \He) \in L^2(\Oe, \C^3) \times L^2(\Oe, \C^3)$ such that
\begin{subequations}\label{eq:weak:Eta:E-H}
  \begin{align}\label{eq:weak:Eta:E}
    \int\limits_{\Oe} \Ee \cdot \curl \psi
    &= \int\limits_{\Oe} (i \omega \mu \He + f_h) \cdot \psi
    && \text{ for every } \psi \in H(\curl, \Oe)\,,
    \\\label{eq:weak:Eta:H}
    \int\limits_{\Oe} \He \cdot \curl \phi
    &= \int\limits_{\Oe} (-i \omega \eps \Ee + f_e) \cdot \phi
    && \text{ for every } \phi \in H_0(\curl, \Oe)\,.
  \end{align}
\end{subequations}
We note that \eqref{eq:weak:Eta:E} does not only imply that
\eqref{eq:Strong:Maxwell:eta:1} is satisfied in the sense of
distributions (and, hence, also the regularity
$\curl \Ee \in L^2(\Oe)$ is satisfied), but it implies also the
boundary condition, $\Ee\in H_0(\curl, \Oe)$ and, hence, \eqref
{eq:Strong:Maxwell:eta:3} in a weak form.

We note the following facts on trivially extended solutions (one
considers $\Ee$ and $\He$ as functions on all of $\Omega$ by setting
them equal to $0$ on $\Sigma_\eta$). The extended functions are
solutions of \eqref{eq:weak:Eta:E-H} with integrals over all of
$\Omega$ under the assumption that $f_h$ vanishes on
$\Sigma_\eta$ and that the test functions are
$\phi, \psi\in H(\curl, \Omega)$ with $\phi$ vanishing on
$\Sigma_\eta$.

\subsection{Weak form of the limit systems} In this section, we state
the weak forms of the limit system \eqref{eq:Maxwell-strong-E-H:Limit}
with one of the interface conditions
\eqref{eq:Maxwell-strong-E-H:Reflecting},
\eqref{eq:Maxwell-strong-E-H:Inactive} or
\eqref{eq:Maxwell-strong-E-H:Polarizing}. The three corresponding
equations have actually all the same form and the same solution space,
they only differ in the spaces of test-functions. We use the
macroscopic geometry given by $\Omega$ and the interface $\Gamma$ of
\eqref{eq:def:Omega-Gamma}.

\begin{subequations}\label{eq:def:Spaces:X-Y}
  We introduce spaces of test-functions.  We start with the largest
  spaces, $X_-$ and $Y_-$, the subscript indicates that there are no
  restrictions on the functions on the interface:
  \begin{align}
    \label{def:X-}
    X_- &\coloneqq C_0^1(\overline{\Omega}, \C^3)\,,\\
    \label{def:Y-}
    Y_- &\coloneqq \left\{ \psi \colon \Omega\to \C^3\mid
          \psi|_{\Omega_\pm} \in C^1(\overline{\Omega_\pm}, \C^3)
          \right\}\,.  
  \end{align}
  To recall some notation: $X_-$ consists of differentiable functions
  $\phi \colon \Omega\to \C^3$ such that $D\phi$ is continuous and bounded
  and such that $\phi=0$ holds on $\del\Omega$. The space $Y_-$
  requires the same properties on $\Omega_+$ and on $\Omega_-$, but
  the functions can have arbitrary values on $\Gamma$ and all
  components can also have a jump across $\Gamma$.

  For an index $i \in \{1,2\}$ that indicates a tangential direction,
  we define subspaces in which some restriction is introduced for the
  $i$-th component of the field:
  \begin{align}
    \label{def:Xi}
    X_i &\coloneqq \{\phi \in X_- \mid \phi_i|_\Gamma = 0\}\,, \\
    \label{def:Yi}
    Y_i &\coloneqq \{\psi \in Y_- \mid
          \llbracket  \psi_i\rrbracket_\Gamma = 0\}\,,
  \end{align}
  where $\llbracket \psi_i\rrbracket_\Gamma$ denotes the jump of
  $\psi_i$. The jump is classically defined for test-functions.

  Finally, we introduce the spaces $X_{12}$ and $Y_{12}$ in which
  we impose a condition for both tangential components:
  \begin{align}
    \label{def:X12}
    X_{12} &\coloneqq X_1 \cap X_2\,,\\
    \label{def:Y12}
    Y_{12} &\coloneqq Y_1 \cap Y_2\,.
  \end{align}
  We mention that
  $X_{12} = \left\{\phi \in C_0^1(\overline{\Omega}, \C^3)\, \mid\, \phi_1 =
    \phi_2 = 0\text{ on } \Gamma\right\}$, and that for functions
  $\psi\in Y_{12}$ the jumps vanish,
  $\llbracket \psi_1\rrbracket_\Gamma = \llbracket
  \psi_2\rrbracket_\Gamma = 0$.  The spaces are ordered:
  $X_{12} \subset X_i \subset X_-$ and
  $Y_{12} \subset Y_i \subset Y_-$.
\end{subequations}

We can now formulate the limit system in a weak sense.  We say that
$E^\hom, H^\hom \in L^2(\Omega, \C^3)$ is a weak solution to
\eqref{eq:Maxwell-strong-E-H:Limit} with interface conditions when
\begin{subequations}\label{eq:weak:Limit-NonComplete:E-H}
  \begin{align}
    \label{eq:weak:Limit-NonComplete:E}
    \int\limits_{\Omega
    \setminus \Gamma} E^\hom\cdot \curl \psi
    &= \int\limits_{\Omega
      \setminus \Gamma} (i \omega \mu H^\hom + f_h) \cdot \psi
    && \text{ for every } \psi \in Y\,, \\
    \label{eq:weak:Limit-NonComplete:H}
    \int\limits_{\Omega} H^\hom \cdot \curl \phi
    &= \int\limits_{\Omega}
      (-i \omega \eps E^\hom + f_e) \cdot \phi
    && \text{ for every } \phi \in  X\,.
  \end{align}
\end{subequations}
The interface conditions \eqref{eq:Maxwell-strong-E-H:Reflecting} are
encoded by the choice $X= X_{12}$ and $Y = Y_-$.  The interface
conditions \eqref{eq:Maxwell-strong-E-H:Inactive} are encoded by the
choice $X= X_-$ and $Y = Y_{12}$.  The interface conditions
\eqref{eq:Maxwell-strong-E-H:Polarizing} are encoded by the choice
$X = X_i$ and $Y = Y_i$. The different cases are collected in
Table~\ref{tab:WeakFormSpaces}. Further comments on the weak solution
concept are given in Remarks \ref {rem:weak-implies-strong} and \ref
{rem:additional-solution-prop} below.

\begin{table}[ht] \centering
  \begin{tabular}{|l|c|c|c|c|c|c|c}
    \hline
    \rule{0pt}{2.5ex} Interface
    & \multicolumn{2}{c|}{strong form}
    &
      \multicolumn{2}{c|}{test-functions} \\[1.5mm]
    \hline
    \rule{0pt}{3ex} Reflecting
    &
      \eqref{eq:Maxwell-strong-E-H:Reflecting}
    &
      $E^\hom_1|_\Gamma = E^\hom_2|_\Gamma = 0$ & $X= X_{12}$ & $Y= Y_-$
    \\[1.5mm] \hline \rule{0pt}{3ex} Inactive &
                                                \eqref{eq:Maxwell-strong-E-H:Inactive}
    &$\llbracket H^\hom_1\rrbracket_\Gamma = \llbracket
      H^\hom_2\rrbracket_\Gamma = 0$ & $X = X_-$
    & $Y=Y_{12}$ \\[1.5mm]
    \hline \rule{0pt}{3ex} Polarizing &
                                        \eqref{eq:Maxwell-strong-E-H:Polarizing}
    &
      $E^\hom_i|_\Gamma = \llbracket  H^\hom_i\rrbracket_\Gamma  = 0$
    & $X = X_i$ & $Y = Y_i$\\[1.5mm]
    \hline
   \end{tabular}
                            
   \caption{Spaces of test functions in
     \eqref{eq:weak:Limit-NonComplete:E-H}. Regarding the third case:
     The index $i\in \{1,2\}$ indicates the direction of
     connectedness.  }
    \label{tab:WeakFormSpaces}
\end{table}

\subsection{Homogenization result} The following theorem collects the
homogenization results. In the homogenized Maxwell system, the
inclusions are replaced by interface conditions along $\Gamma$. These
conditions depend on the asymptotic connectedness and disconnectedness
of the inclusions as specified in Definitions
\ref{def:Analytical-DisConnectedness} and
\ref{def:Analytical-Connectedness}.

\begin{theorem}[Main result]\label{thm:main:homogenization} Let the
  setting be that of Section~\ref{sec:Geometry}: The macroscopic sets
  are $\Omega$ and $\Gamma$, the sets for the $\eta$-problem are given by
  $\Sigma^\eta_Y$, $\Oe$ and $\Sigma_\eta$ for a sequence
  $\eta \to 0$.  Let
  $(\Ee, \He) \in L^2(\Oe, \C^3) \times L^2(\Oe, \C^3)$ be solutions
  to \eqref{eq:weak:Eta:E-H}. We assume that the trivial extensions of
  $\Ee$ and $\He$ to $\Omega$ have weak limits,
  \begin{align*} \Ee \weakto E^\hom\,, \qquad \He \weakto H^\hom
    \qquad \text{weakly in } L^2(\Omega, \C^3)\,.
  \end{align*} Then, the equations for $E^\hom$ and $H^\hom$ are given as
  follows, depending on the connectedness properties of the
  micro-structure:
  
  \smallskip
  \noindent {Case 1: {\bf Reflecting}}\quad When $\Sigma^\eta_Y$ is
  asymptotically connected in both directions $e_1$ and $e_2$, then
  $(E^\hom, H^\hom)$ solves \eqref{eq:Maxwell-strong-E-H:Limit} and
  \eqref{eq:Maxwell-strong-E-H:Reflecting} weakly, ie
  \eqref{eq:weak:Limit-NonComplete:E-H} holds for $X= X_{12}$ and
  $Y =Y_-$.

  \smallskip
  \noindent {Case 2: {\bf Inactive}}\quad When $\Sigma^\eta_Y$ is
  asymptotically disconnected in both directions $e_1$ and $e_2$, then
  $(E^\hom, H^\hom)$ solves \eqref{eq:Maxwell-strong-E-H:Limit} and
  \eqref{eq:Maxwell-strong-E-H:Inactive} weakly, ie
  \eqref{eq:weak:Limit-NonComplete:E-H} holds for $X= X_-$ and
  $Y =Y_{12}$.
  
  \smallskip
  \noindent {Case 3: {\bf Polarizing}}\quad Let $i \in \{1,2\}$. If
  $\Sigma^\eta_Y$ is asymptotically connected in the direction $e_i$
  and asymptotically disconnected in the direction $e_j$ for
  $j = 3-i$, then $(E^\hom, H^\hom)$ solves
  \eqref{eq:Maxwell-strong-E-H:Limit} and
  \eqref{eq:Maxwell-strong-E-H:Polarizing} weakly, ie
  \eqref{eq:weak:Limit-NonComplete:E-H} holds for $X= X_i$ and
  $Y =Y_i$.
\end{theorem}

\begin{lemma}[Trivial part of the limit system]
  \label{lem:Homogenization-away-Interface} Let the setting be that of
  Theorem~\ref{thm:main:homogenization}, no assumptions on
  connectedness of $\Sigma_Y^\eta$. Then, the limit fields
  $(E^\hom, H^\hom)$ satisfy \eqref{eq:weak:Limit-NonComplete:E-H} for
  $X = X_{12}$ and $Y = Y_{12}$.
\end{lemma}

\begin{proof}
  Let $\psi \in Y_{12}$ be arbitrary, we recall that this implies, in
  particular, that the functions $\psi_1$ and $\psi_2$ are continuous
  across $\Gamma$. The function $\psi$ can therefore be interpreted
  (by restriction) as an element $\psi\in H(\curl, \Omega_\eta)$ --
  jumps in the normal component do not contribute to the curl.  We can
  therefore use $\psi$ as a test-function in
  \eqref{eq:weak:Eta:E}. Passing to the limit provides
  \eqref{eq:weak:Limit-NonComplete:E} for $\psi$.

  \smallskip Let now $\phi \in X_{12}$ be arbitrary, we recall that
  this implies $\phi_1 = \phi_2 = 0$ on $\Gamma$. We approximate
  $\phi$ by a sequence of functions $\phi^k$; for the first two
  components we demand
  $\phi^k_1, \phi^k_2 \in C^\infty_c(\Omega\setminus \Gamma, \C^3)$.
  Regarding the third component, we choose $\phi^k_3 = \phi_3$ on
  $\{ x\in \Omega \mid x_3\not\in (0, \eta) \}$ and $\phi^k_3 = 0$ on
  $\{ x\in \Omega \mid x_3\in (0, \eta) \}$.  We can find an
  approximation with an error
  $\| \phi^k_i - \phi_i\|_{H^1(\Omega)} \le 1/k$ for $i \in \{1,2\}$.
  We note that this also implies
  $\| \curl \phi^k - \curl \phi\|_{L^2(\Omega)} \to 0$ as
  $k\to \infty$.
  
  Since $\Sigma_\eta$ is contained in an $\eta$-neighborhood of
  $\Gamma$, there holds $\phi^k \in H_0(\curl, \Oe)$ for sufficiently
  small $\eta>0$.  We can therefore use $\phi^k$ in
  \eqref{eq:weak:Eta:H}. Passing to the limit $\eta \to 0$, we obtain
  \eqref{eq:weak:Limit-NonComplete:H} with $\phi^k$ instead of $\phi$.
  Taking the limit $k\to \infty$, we obtain
  \eqref{eq:weak:Limit-NonComplete:H} for $\phi$.
\end{proof}

\begin{remark}[Compactly contained obstacles are disconnected]
  \label{rem:compact-inclusions}
  Let all sets $\Sigma_Y^\eta$ be contained in a closed set
  $\Sigma_Y^0\subset Y = (0,1)^3$ that is independent of $\eta$. Then,
  $\Sigma^\eta_Y$ is asymptotically disconnected in both directions
  $e_1$ and $e_2$. Accordingly, we are in the situation of Case 2 of
  Theorem \ref{thm:main:homogenization}: The micro-structure has no
  effect.
\end{remark}

The proof of Remark \ref {rem:compact-inclusions} is provided in the
end of Section \ref {sec:wires}.

\begin{theorem}[Main result for wires]
  \label{thm:main-wires} In the situation of Theorem
  \ref{thm:main:homogenization}, let the inclusions be given by wires
  in direction $i = 1$, $\Sigma_Y^\eta = T_{r_\eta, I_\eta}$ of
  \eqref{def:T_I,r}. We assume that the radii $r_\eta$ of the wires
  are not too small and that the slits $I_\eta$ are not too big in the
  sense that:
  \begin{equation*} \eta |\ln(r_\eta)| \to 0 \qquad \textrm{ and }
    \qquad \eta^{-1} r_\eta^{-2} |I_\eta| \to 0 \,.
  \end{equation*}
  Then, the limit system is given by Theorem
  \ref{thm:main:homogenization}, Case 3: $E^\hom, H^\hom$ solves
  \eqref{eq:Maxwell-strong-E-H:Limit} and
  \eqref{eq:Maxwell-strong-E-H:Polarizing} with $i = 1$.
\end{theorem}

\begin{proof} The wires are connecting in direction $e_1$ by
  Proposition \ref {prop:construction:Psi-wires-perfect}.  The wires
  are disconnected in direction $e_2$ by Proposition \ref
  {prop:construction:Phi-wires-ortho}. Theorem
  \ref{thm:main:homogenization} therefore yields the limit system with
  Case 3.
\end{proof}

\begin{remark}[On the weak solution concept]
  \label{rem:weak-implies-strong}
  The above choice of weak equations yields indeed a weak solution
  concept for \eqref{eq:Maxwell-strong-E-H:Limit} with interface
  condition. We verify this fact in the polarizing case with $i =
  1$. Let $(E^\hom, H^\hom)$ be a regular solution of \eqref
  {eq:weak:Limit-NonComplete:E-H} with $X = X_1$ and $Y = Y_1$, the
  solution is required to be regular enough to allow for classical
  integration by parts.  As a first step, we use an arbitrary
  test-function $\phi\in C_c^\infty(\Omega\setminus\Gamma, \C^3)$ in
  \eqref {eq:weak:Limit-NonComplete:H} and obtain
  \eqref{eq:Maxwell-strong-E-H:Limit:2}.  Similarly, we can use an
  arbitrary test-function $\psi\in C_c^\infty(\Omega, \C^3)$ in \eqref
  {eq:weak:Limit-NonComplete:E} and obtain
  \eqref{eq:Maxwell-strong-E-H:Limit:1}.

  Let $\phi\in X_1$ be an arbitrary test-function. We calculate with
  \eqref {eq:weak:Limit-NonComplete:H}:
  \begin{align*}
    0  &= \int\limits_{\Omega_+} H^\hom \cdot \curl \phi
         + (i \omega \eps E^\hom - f_e) \cdot \phi
         + \int\limits_{\Omega_-} H^\hom \cdot \curl \phi
         + (i \omega \eps E^\hom - f_e) \cdot \phi\\
       &= \int\limits_{\Omega} (\rot H^\hom + i \omega \eps E^\hom - f_e) \cdot \phi
         + \int\limits_{\Gamma} \llbracket H^\hom \rrbracket_\Gamma \cdot \phi \times e_3
         = \int\limits_{\Gamma} \llbracket H^\hom \rrbracket_\Gamma \cdot \phi \times e_3\,.
  \end{align*}
  Since the values of $\phi_2$ on $\Gamma$ can be chosen arbitrarily,
  we can conclude from our calculation
  $\llbracket H^\hom \rrbracket_\Gamma \cdot e_1 = 0$, which provides
  the second relation of \eqref{eq:Maxwell-strong-E-H:Polarizing}.

  Similarly, we can calculate for an arbitrary test-function
  $\psi\in Y_1$ with \eqref {eq:weak:Limit-NonComplete:E}:
  \begin{align*}
    0  &= \int\limits_{\Omega_+} E^\hom \cdot \curl \psi
         - (i \omega \mu H^\hom + f_h) \cdot \psi
         + \int\limits_{\Omega_-} E^\hom \cdot \curl \psi
         - (i \omega \mu H^\hom + f_h) \cdot \psi\\
       &= \int\limits_{\Omega} (\curl E^\hom
         - i \omega \mu H^\hom + f_h) \cdot \psi
         + \int\limits_{\Gamma} E^\hom\cdot \llbracket \psi \rrbracket_\Gamma \times e_3
         = \int\limits_{\Gamma} E^\hom\cdot \llbracket \psi \rrbracket_\Gamma \times e_3\,.
  \end{align*}
  Since the jump $\llbracket \psi \rrbracket_\Gamma$ can have an
  arbitrary second component, our calculation implies
  $E^\hom \cdot e_1 = 0$ on $\Gamma$, and hence the first relation of
  \eqref{eq:Maxwell-strong-E-H:Polarizing}.
\end{remark}

\begin{remark}[$L^2$-boundedness of the solution sequence]
  \label{rem:L2-bound}
  The statement of Theorem~\ref {thm:main:homogenization} assumes the
  boundedness of the solution sequence
  $(\Ee, \He) \in L^2(\Oe, \C^3) \times L^2(\Oe, \C^3)$.  The
  derivation of this boundedness is non-trivial due to the possible
  presence of localized (surface) waves along the microstructure. We
  plan to investigate such resonance type phenomena in a future
  publication. The present work deals with the homogenization of the
  system, which requires the derivation of equations for weak limits.
\end{remark}

\begin{remark}[Density, symmetry properties of the limit system]
  \label{rem:additional-solution-prop}
  One can expect density of the test-functions in all cases of
  Table~\ref{tab:WeakFormSpaces}.  More precisely: $X_{12}$ is dense
  in the space of functions $E\in H(\curl, \Omega)$ with
  $E_1|_\Gamma = E_2|_\Gamma = 0$. The space $Y_{12}$ is dense in the
  space of functions $H\in H(\curl, \Omega)$ with
  $\llbracket H_1\rrbracket_\Gamma = \llbracket H_2\rrbracket_\Gamma =
  0$. Analogously for $X_i$ and $Y_i$ and $X_-$ and $Y_-$.  In all
  cases except for $X_i$ and $Y_i$, these are actually classical
  results.

  Once the density is verified, the limit system has actually a
  symmetric form in solutions and test-functions. Let us formulate
  this statement in the polarizing case: Test functions are
  $\phi\in\overline X_i$ and $\psi\in\overline Y_i$, the closures are
  taken in $H(\curl, \Omega)$. The solution is in the same space,
  $E^\hom\in\overline X_i$ and $H^\hom\in\overline Y_i$. Such a
  symmetric formulation is probably useful in the derivation of
  existence and uniqueness results.
\end{remark}

\section{Cell functions and asymptotic
(dis-)connectedness}\label{sec:Test-functions} The homogenization of
\eqref{eq:weak:Eta:E-H} is based on an appropriate choice of
test-functions. We will actually {\em define} connectedness properties
of $\Sigma_\eta$ with the existence of sequences of functions in a
cell-geometry; the sequence of functions, if it exists, is used to
construct test-functions for the original problem in $\Omega_\eta$.

To motivate these cell functions, let us describe the homogenization
procedure: Let $i \in \{1,2\}$ and $j= 3-i$.  Lemma
\ref{lem:Homogenization-away-Interface} shows
\eqref{eq:weak:Limit-NonComplete:E} for $Y= Y_{12}$. This space does
not provide information on $E^\hom$ at the interface $\Gamma$.  Let us
assume that we want to use the larger set of test-functions $Y =
Y_i$. Functions in $Y_i$ have no jump in the $i$-th component, but
they can have a jump in the $j$-th component. Thus, we cannot use them
in the $\eta$-problem \eqref{eq:weak:Eta:E}. Instead, we consider a
sequence of functions $\Psi_\eta^{(i)}$ that approximate a jump in the
$j$-th component at $\Gamma$. Passing to the limit, we obtain
\eqref{eq:weak:Limit-NonComplete:E} for $Y = Y_i$. If such a sequence
$\Psi_\eta^{(i)}$ exists, we say that $\Sigma^\eta_Y$ is connected in
the direction $e_i$, see
Definition~\ref{def:Analytical-Connectedness}.

For \eqref{eq:weak:Limit-NonComplete:H} the procedure is
similar. Lemma \ref{lem:Homogenization-away-Interface} shows
\eqref{eq:weak:Limit-NonComplete:H} for $X= X_{12}$, let us assume
that we want to use $X = X_i$.  The $j$-component of a function in
$X_i$ does not vanish on $\Gamma$, we cannot use it directly as a
test-function. But let us assume that there is a sequence of functions
$\Phi_\eta^{(j)}$, vanishing in $\Sigma_\eta$ and approximating, for
$\eta \to 0$, functions that do not vanish in the $j$-th component at
$\Gamma$. With these test-functions, we can pass to the limit and
obtain \eqref{eq:weak:Limit-NonComplete:H} for $X = X_i$. If such a
sequence $\Phi_\eta^{(j)}$ exists, we say that $\Sigma^\eta_Y$ is
disconnected in the direction $e_j$, see
Definition~\ref{def:Analytical-DisConnectedness}.

\begin{figure}[ht]
  \centering
  \begin{subfigure}{.49\textwidth}
    \centering
    \includegraphics[width=0.99\linewidth]{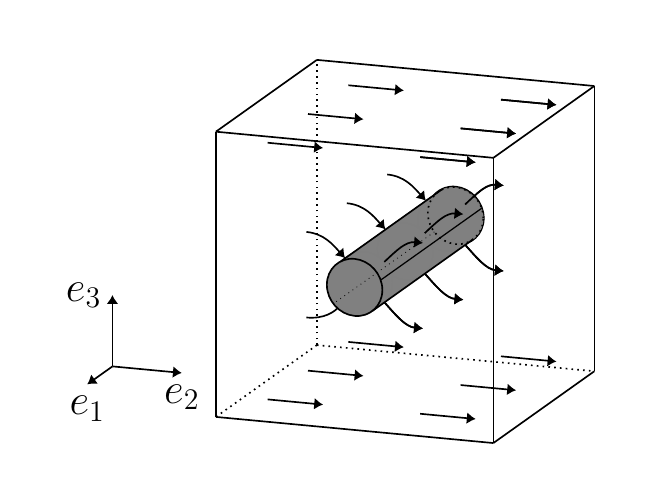}
    \subcaption{The cell function $\Phi_\eta^{(2)}$ shows that the
      wire is disconnected in direction $e_2$}
    \label{fig:Sigma-Slit}
  \end{subfigure}
  \hspace{\fill}
  \begin{subfigure}{.49\textwidth}
    \centering
    \includegraphics[width=0.99\columnwidth]{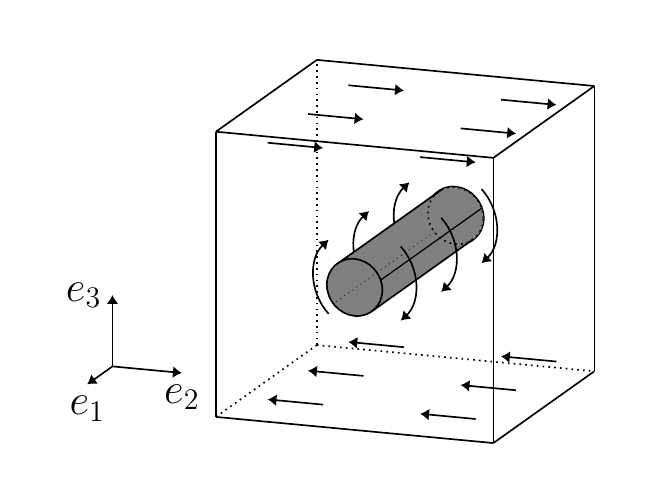}
    \subcaption{The cell function $\Psi_\eta^{(1)}$ shows that the
      wire is connected in direction $e_1$}
  \end{subfigure}
  \caption{Visualization of the cell functions in $Y$ for $i=1$, $j=2$
    and $\Sigma^\eta_Y = T^{(1)}_{0.15}\,$.  The illustration in (B)
    does not show the cell function $\Psi_\eta^{(1)}$, but sketches
    $2\Psi_\eta^{(1)} - e_2$. }
  \label{fig:Cellfunction}
\end{figure} 

\subsection*{The cell functions}
 The cell functions are defined in the cylinder $Z$ and we consider
$Z^\eta \subset Z$ given by
\begin{equation*}
  Z \coloneqq (0,1)^2 \times \R\,, \qquad Z^\eta
  \coloneqq Z \setminus \Sigma^\eta_Y\,.
\end{equation*}
We define the function space
\begin{align*} &\Hd(\rot,Z) \coloneqq \{u|_{Z}\mid u \in
H_\loc(\rot,\R^3)\text{ is } e_1\text{- and } e_2\text{-periodic}\}
\,.
\end{align*} We always identify a function $\Hd(\rot,Z)$ with its
periodic extension to $\R^3$.

\subsection{\texorpdfstring{$E$}{E}-type cell functions and asymptotic
  disconnectedness} Let $j \in \{1,2\}$, we use cell functions
$\Phi^{(j)}_\eta \in \Hd(\rot,Z)$ that converge to $e_j$ for
$|x_3| \to \infty$, but vanish on $\Sigma^\eta_Y$. Important is an
asymptotic quantitative control of the $\curl$.

\begin{definition}[Asymptotically disconnected obstacles]
  \label{def:Analytical-DisConnectedness} Let $\Sigma^\eta_Y$ be a
  sequence of obstacles as in Section \ref{sec:Geometry}. For
  $j \in \{1,2\}$, we say that the sequence $\Sigma^\eta_Y$ is
  asymptotically disconnected in the direction $e_j$ if there exists a
  sequence of cell functions $\Phi^{(j)}_\eta \in \Hd(\rot,Z)$ of the
  $E$-type with direction $e_j$ as follows:
  \begin{subequations}\label{eq:CellProblem:Phi}
    \begin{align} \label{eq:CellProblem:Phi:L2} \eta^{\tfrac{1}{2}}
      \|\Phi^{(j)}_\eta - e_j \|_{L^2(Z)} &\to 0 \,,\\
      \label{eq:CellProblem:Phi:obstacle} \Phi^{(j)}_\eta
      |_{\Sigma^\eta_Y} &= 0 \,, \\
	\label{eq:CellProblem:Phi:curl} \eta^{-\tfrac{1}{2}} \|\curl
      \Phi^{(j)}_\eta\|_{L^2(Z)} &\to 0\,.
    \end{align}
  \end{subequations}
\end{definition}

\subsection{\texorpdfstring{$H$}{H}-type cell functions and asymptotic
  connectedness} We are interested in cell functions
$\Psi^{(i)}_\eta \in \Hd(\rot,Z)$ that converge to a direction $e_1$
or $e_2$ for $x_3 \to \infty$, but to $0 $ for $x_3 \to
-\infty$. Important is an asymptotic quantitative control of the
$\curl$.

\begin{definition}[Asymptotically connecting obstacles]
  \label{def:Analytical-Connectedness} Let $\Sigma^\eta_Y$ be a
  sequence of obstacles as in Section \ref{sec:Geometry}. For
  $i \in \{1,2\}$ and $j = 3-i$, we say that the sequence $\Sigma^\eta_Y$ is
  asymptotically connecting in the direction $e_i$ if there exists a
  sequence of cell functions $\Psi^{(i)}_\eta \in \Hd(\rot,Z)$ of the
  $H$-type with direction $e_i$ as follows:
  \begin{subequations}\label{eq:CellProblem:Psi}
    \begin{align}
      \label{eq:CellProblem:Psi:L2} \eta^{\tfrac{1}{2}}
      \|\Psi_\eta^{(i)} - e_j \one{\{x_3 >0\}}\|_{L^2(Z^\eta)} &\to 0 \,, \\
      \label{eq:CellProblem:Psi:curl} \eta^{-\tfrac{1}{2}} \|\curl
      \Psi_\eta^{(i)}\|_{L^2(Z^\eta)} &\to 0\,.
    \end{align}
  \end{subequations}
\end{definition}

\begin{remark}[Cell functions pointing in direction $e_3$]
  \label{rem:Trivial:Phi-Psi} We note that
  \eqref{eq:CellProblem:Phi} for $j=3$ has the trivial solution
  $e_3 \one{\{x_3>1\} \cup \{x_3<0\}}$. Similarly,
  \eqref{eq:CellProblem:Psi} for $j=3$ has the trivial solution
  $e_3 \one{\{x_3>1\}}$.
\end{remark}

\subsection{Rescaling of the cell functions} To employ the above
introduced cell functions for homogenization, we scale them by $\eta$
and consider, e.g.~the function $\Phi^{(j)}_\eta(\cdot/\eta)$. The
subsequent lemma collects the properties of the rescaled functions.

\begin{lemma}[$\eta$-scaling of the cell
  functions]\label{lem:Extension-to-Omega} Let
  $\Phi^{(j)}_\eta \in \Hd(\rot,Z)$ be a sequence that satisfies
  \eqref{eq:CellProblem:Phi}. Then,
  $\Phi^{(j)}_\eta(\cdot/\eta) \in H(\curl, \Omega)$ is a sequence
  with the property $\Phi^{(j)}_\eta(\cdot/\eta)|_{\Sigma_\eta} =0$
  and with the convergences
  \begin{subequations}\label{eq:Extension-to-Omega:Phi}
    \begin{align}\label{eq:Extension-to-Omega:Phi:1}
      \Phi^{(j)}_\eta(\cdot/\eta)
      &\to e_j&& \textrm{ in } L^2(\Omega) \, ,
      \\\label{eq:Extension-to-Omega:Phi:2} \curl
      (\Phi^{(j)}_\eta(\cdot/\eta))
      & \to 0&& \textrm{ in } L^2(\Omega) \, .
    \end{align}
  \end{subequations}
  
  Let $\Psi^{(i)}_\eta \in \Hd(\rot,Z)$ be a sequence that satisfies
  \eqref{eq:CellProblem:Psi} for $j = 3-i$.  Then,
  $\Psi^{(i)}_\eta(\cdot/\eta) \in H(\curl, \Omega)$ and
  \begin{subequations}\label{eq:Extension-to-Omega:Psi}
    \begin{align}\label{eq:Extension-to-Omega:Psi:1}
      \one{\Oe}(\cdot)\Psi^{(i)}_\eta(\cdot/\eta)
      &\to e_j \one{\{x_3 >0\}}(\cdot)&& \textrm{ in } L^2(\Omega) \, ,
      \\\label{eq:Extension-to-Omega:Psi:2}
      \one{\Oe}(\cdot) \curl (\Psi^{(i)}_\eta(\cdot/\eta))
      & \to 0&& \textrm{ in } L^2(\Omega) \, .
    \end{align}
  \end{subequations}
\end{lemma}
      
\begin{proof}
  The calculations for \eqref {eq:Extension-to-Omega:Phi} and \eqref
  {eq:Extension-to-Omega:Psi} are almost identical. We present the
  calculations for \eqref {eq:Extension-to-Omega:Psi}, since these are
  slightly more interesting due to the presence of characteristic
  functions in the limit.
  
  The set of relevant indices $k = (k_1, k_2)$ is
  $K_\eta \coloneqq \{ k\in \Z^2 \mid \eta (k + Z) \cap \Omega \neq
  \emptyset\}$.  The number of elements of $K_\eta$ is of the order
  $\eta^{-2}$; since we do not assume $\eta^{-1}\in \N$, we have
  $|K_\eta| \leq (\eta^{-1}+1)^2$ and, for $\eta$ small,
  $|K_\eta| \leq 2\eta^{-2}$.  The convergence
  \eqref{eq:Extension-to-Omega:Psi:1} is obtained from
  \eqref{eq:CellProblem:Psi:L2} with the following computation:
  \begin{align*}
    &\left\|\one{\Oe} \Psi^{(i)}_\eta(\cdot/\eta) -e_j \one{\{x_3 >0\}} \right\|_{L^2(\Omega)}^2
      \leq \sum\limits_{k \in K_\eta} \left\|\one{\Oe}\Psi^{(i)}_\eta(\cdot/\eta) -e_j \one{\{x_3 >0\}} \right\|_{L^2(\eta k + \eta Z)}^2
    \\
    &\qquad = \eta^3 \left(\sum\limits_{k \in K_\eta} \left\|\Psi^{(i)}_\eta -e_j \one{\{x_3 >0\}} \right\|_{L^2(Z^\eta)}^2 + \left\|e_j \right\|_{L^2(\Sigma^\eta_Y)}^2 \right)
    \\
    &\qquad \leq \eta^3\, 2\eta^{-2} \left(\|\Psi^{(i)}_\eta -e_j\one{\{x_3 >0\} } \|_{L^2(Z^\eta)}^2 + \left\|e_j \right\|_{L^2(\Sigma^\eta_Y)}^2 \right)
      \to 0 \, .
  \end{align*}
  Using
  $\curl (\Psi^{(i)}_\eta(\cdot/\eta)) = \eta^{-1}(\curl
  \Psi^{(i)}_\eta)(\cdot/\eta)$ and
  \eqref{eq:CellProblem:Psi:curl}, a similar computation
  gives:
  \begin{align*}
    &\|\one{\Oe}(\curl \Psi^{(i)}_\eta(\cdot/\eta))\|_{L^2(\Omega)}^2 
      \leq \sum\limits_{k \in K_\eta} \left\|\eta^{-1} (\curl \Psi^{(i)}_\eta)(\cdot/\eta) \right\|_{L^2(\eta k + \eta Z^\eta)}^2\\
    &\qquad = \eta^{3} \sum\limits_{k \in K_\eta} \|\eta^{-1} \rot \Psi^{(i)}_\eta\|_{L^2(Z^\eta)}^2
      = \eta^{3}\, 2\eta^{-2} \|\eta^{-1} \rot \Psi^{(i)}_\eta\|_{L^2(Z^\eta)}^2 \\
    &\qquad = \eta^{-1}\, 2\|\rot \Psi^{(i)}_\eta\|_{L^2(Z^\eta)}^2 \to 0 \,.
  \end{align*}
  This shows \eqref{eq:Extension-to-Omega:Psi:2}.
\end{proof}
 
\section{Homogenization of Maxwell's equations}
\label{sec:Hom-Maxwell}
In this section, we pass to the limit $\eta \to 0$ in
\eqref{eq:Strong:Maxwell:eta} and prove Theorem
\ref{thm:main:homogenization}. The theorem states that system
\eqref{eq:weak:Limit-NonComplete:E-H} holds for test-functions
$\phi\in X$ and $\psi\in Y$, with specific spaces $X$ and $Y$. By
Lemma \ref{lem:Homogenization-away-Interface}, system
\eqref{eq:weak:Limit-NonComplete:E-H} is satisfied for test-functions
$\phi\in X_{12}$ and $\psi\in Y_{12}$. 

Let us recall the principle of our notation for test-functions: The
spaces $X_-$ and $Y_-$ have no condition on $\Gamma$, the spaces $X_i$
and $Y_i$ contain a condition on the $i$-th component on $\Gamma$, the
spaces $X_{12}$ and $Y_{12}$ contain conditions on both tangential
components on $\Gamma$.  When we want to show, e.g.~that an equation
holds for all $\phi\in X_1$, we show that: (a) the equation holds for
all $\phi\in X_{12}$, and (b) the equation holds for all functions
$\phi\in C^1_0(\overline{\Omega}, \C\, e_2)$.

\subsection{Limits of \texorpdfstring{$\Ee$}{E-eta}-fields at the
  interface \texorpdfstring{$\Gamma$}{Gamma}}
\label{ssec:Limit-Ee}
Here, we want to verify \eqref{eq:weak:Limit-NonComplete:E} for
certain test-functions. We can restrict ourselves to the $E$-equation
and study the following situation: Let $\Ee$ and $f^\eta$ be two
sequences in $L^2(\Omega, \C^3)$ with
\begin{subequations}\label{eq:eta:curl-E=f}
  \begin{align}\label{eq:eta:curl-E=f:1}
    \int\limits_{\Omega} \Ee \cdot \rot \psi
    &= \int\limits_{\Omega} f^\eta \cdot \psi \,
    &&\text{for every }\psi \in  H(\curl,\Omega) \, ,
    \\\label{eq:eta:curl-E=f:2}
    E^\eta &= 0&& \text{on } \Sigma_\eta\,.
  \end{align} 
\end{subequations}
We recall that this can be regarded as a weak formulation of the
equation $\rot \Ee = f^\eta$ in $\Omega_\eta$ and
$\Ee \in H_0(\curl, \Omega_\eta)$. In particular, by
\eqref{eq:eta:curl-E=f:2}, $f^\eta|_{\Sigma_\eta} = 0$.  Note
that \eqref{eq:eta:curl-E=f:1} corresponds to
\eqref{eq:weak:Eta:E} for
$f^\eta \coloneqq \one{\Oe}(i \omega \mu H^\eta + f_h)$.

\begin{proposition}[Interface condition for the limit $E$]
  \label{prop:Trace-Ei}
  Let $\Omega$, $\Sigma^\eta_Y$, $\Sigma_\eta$ and $\Gamma$ be given
  as in Section \ref{sec:Geometry}, let $i \in \{1,2\}$ and $j = 3-i$
  indicate orthogonal directions. We assume that $\Sigma^\eta_Y$ is
  asymptotically connected in the direction $e_i$ in the sense of
  Definition~\ref{def:Analytical-Connectedness}, i.e.~there is a
  sequence $\Psi^{(i)}_\eta$ satisfying
  \eqref{eq:CellProblem:Psi}. Let $\Ee$ and $f^\eta$ be sequences
  satisfying \eqref{eq:eta:curl-E=f} that converge weakly in
  $L^2(\Omega, \C^3)$ to limit functions $E$ and $f$, respectively.
  Then, the limits satisfy, for both $+$ and $-$:
  \begin{align}\label{eq:weak:Interface:Ei}
    \int\limits_{\Omega_\pm} E \cdot \curl \psi 
    &= \int\limits_{\Omega_\pm} f \cdot \psi 
    \quad \text{ for every } \psi \in C^1(\overline{\Omega_\pm}, e_j\, \C)\,.
  \end{align}
  We recall that \eqref{eq:weak:Interface:Ei} encodes
  $E_i|_\Gamma = 0$ (for both traces, from $\Omega_+$ and from
  $\Omega_-$).
\end{proposition}

\begin{proof}
  We present the proof for $\Omega_+$. Let
  $\psi \in C^1(\overline{\Omega_+}, e_j\, \C)$ be arbitrary, we write
  the function as $\psi = \fhi\, e_j$ with
  $\varphi \in C^1(\overline{\Omega_+}, \C)$.

  We extend $\varphi$ to $\Omega$ by reflection across $\Gamma$,
  i.e.~$\varphi(x_1,x_2,x_3) = \varphi(x_1,x_2,-x_3)$, respectively,
  for $x \in \Omega_-$. The extended function $\varphi$ is continuous
  across $\Gamma$ and of class $H^1(\Omega)$.  Let $\Psi^{(i)}_\eta$
  be the cell function of \eqref{eq:CellProblem:Psi}. We note that the
  strong convergences \eqref{eq:Extension-to-Omega:Psi} imply the
  strong $L^2(\Omega)$-convergence
  \begin{align}
    \begin{aligned}\label{eq:strong-convergence:Product:Psi}
      \one{\Oe}(\cdot)\rot\big(\varphi(\cdot) \Psi^{(i)}_\eta(\cdot
      /\eta) \big)&= \varphi\one{\Oe} (\curl
      \Psi^{(i)}_\eta(\cdot/\eta)) + \nabla \varphi \times
      \Psi^{(i)}_\eta(\cdot/\eta)\one{\Oe}
      \\
      &\to \nabla \varphi \times e_j \one{\{x_3 >0\}} = \rot(\varphi
      e_j) \one{\{x_3 >0\}}\,.
    \end{aligned}
  \end{align}
  Because of
  $\varphi(\cdot) \Psi^{(i)}_\eta(\cdot/\eta) \in H(\curl,\Omega)$, we
  can use this test-function in \eqref{eq:eta:curl-E=f:1}.  Since
  $\Ee$ and $f^\eta$ vanish in $\Oe$, we can multiply the integrand of
  \eqref{eq:eta:curl-E=f:1} also with $\one{\Oe}$.  We calculate for
  the limit $\eta\to 0$:
  \begin{align*}
    0 &= \int\limits_{\Omega} f^\eta(x) \cdot
        \varphi(x) \one{\Oe}(x)\Psi^{(i)}_\eta(x/\eta) - \Ee(x) \cdot
        \one{\Oe}(x)\rot\big(\varphi(x) \Psi^{(i)}_\eta(x/\eta) \big)
        \dx
        \displaybreak[2]\\
      &\to \int\limits_{\Omega} f(x) \cdot \varphi(x) e_j \one{\{x_3
        >0\}}(x) - E(x) \cdot \rot(\varphi(x) e_j) \one{\{x_3 >0\}}(x)
        \dx
        \displaybreak[2]\\
      &= \int\limits_{\Omega_+} f \cdot \varphi\, e_j - E \cdot
        \rot(\varphi\, e_j)
        = \int\limits_{\Omega_+} f \cdot \psi - E \cdot
        \rot(\psi)\,.
  \end{align*}
  This was the claim.
\end{proof}

\subsection{Limits of \texorpdfstring{$\He$}{H-eta}-fields at the
  interface \texorpdfstring{$\Gamma$}{Gamma}}\label{ssec:Limit-He}
We turn to the verification of \eqref{eq:weak:Limit-NonComplete:H} for
test-functions $\phi$.  As in the last subsection, we consider only
the relevant equation. Here, we consider sequences $H^\eta$ and
$f^\eta$ in $L^2(\Omega, \C^3)$ with
\begin{align}\label{eq:eta:curl-H=f:1}
  \int\limits_{\Omega} \He \cdot \rot \phi
  &= \int\limits_{\Omega} f^\eta \cdot \phi  \,
    \quad \text{for every }\phi \in H_0(\curl, \Omega) \textrm{ with }
    \phi|_{\Sigma_\eta} = 0\,.
\end{align} 
We recall that this is a weak formulation of $\rot \He = f^\eta$ in
$\Omega_\eta$.  Note that \eqref{eq:eta:curl-H=f:1} is
\eqref{eq:weak:Eta:H} for
$f^\eta \coloneqq \one{\Oe}(-i \omega \eps E^\eta + f_e)$.

\begin{proposition}[Interface condition for the limit $H$]
  \label{prop:Trace-Hi}
  Let $\Omega$, $\Sigma^\eta_Y$, $\Sigma_\eta$ and $\Gamma$ be given
  as in Section \ref{sec:Geometry}, let $i \in \{1,2\}$ and $j = 3-i$
  indicate orthogonal directions. We assume that $\Sigma^\eta_Y$ is
  asymptotically disconnected in the direction $e_j$ in the sense of
  Definition~\ref{def:Analytical-DisConnectedness}, i.e.~there is a
  sequence $\Phi^{(j)}_\eta$ satisfying
  \eqref{eq:CellProblem:Phi}. Let $\He$ and $f^\eta$ be sequences
  satisfying \eqref{eq:eta:curl-H=f:1}, with $\He\weakto H$ and
  $f^\eta\weakto f$, weakly in $L^2(\Omega, \C^3)$.  Then, the limits
  satisfy
  \begin{align}\label{eq:weak:Interface:Hi}
    \int\limits_{\Omega} H \cdot \curl \phi &= \int\limits_{\Omega} f \cdot \phi
    \quad \text{ for every } \phi \in C^1_0(\overline{\Omega}, e_j \C)\,.
  \end{align}
  We recall that \eqref{eq:weak:Interface:Hi} encodes, in the weak
  sense, the condition $\llbracket H_i\rrbracket_\Gamma = 0$.
\end{proposition}

\begin{proof}
  Let $\phi \in C^1_0(\overline{\Omega}, e_j \C)$ be arbitrary, we write it
  as $\phi = \fhi\, e_j$ for $\varphi \in C^1_0(\overline{\Omega}, \C)$. Let
  $\Phi^{(j)}_\eta$ be the cell function of
  \eqref{eq:CellProblem:Phi}.  The strong convergences
  \eqref{eq:Extension-to-Omega:Phi} imply the strong $L^2(\Omega)$
  convergence
  \begin{align}
    \begin{aligned}\label{eq:strong-convergence:Product:Phi}
      \rot\big(\varphi(\cdot) \Phi^{(j)}_\eta(\cdot /\eta) \big)&=
      \varphi \curl (\Phi^{(j)}_\eta(\cdot/\eta)) + \nabla \varphi
      \times \Phi^{(j)}_\eta(\cdot/\eta)
      \\
      &\to \nabla \varphi \times e_j = \rot(\varphi\, e_j)\,.
    \end{aligned}
  \end{align}
  Because of
  $\varphi(\cdot) \Phi^{(j)}_\eta(\cdot/\eta) \in H_0(\curl,\Oe)$, we
  can use this test-function in \eqref{eq:eta:curl-H=f:1}.  We
  calculate for the limit $\eta\to 0$:
  \begin{align*}
    0 &= \int\limits_{\Omega} f^\eta(x) \cdot
        \varphi(x)\Phi^{(j)}_\eta(x/\eta) - \He(x) \cdot
        \rot\big(\varphi(x) \Phi^{(j)}_\eta(x/\eta) \big) \dx
    \\
      &\to
        \int\limits_{\Omega} f(x) \cdot \varphi(x) e_j - H(x) \cdot
        \rot(\varphi(x) e_j) \dx
        = \int\limits_{\Omega} f \cdot \phi - H \cdot \rot(\phi) 
        \,.
  \end{align*}
  This was the claim.
\end{proof}

\smallskip

\begin{proof}[Proof of Theorem \ref{thm:main:homogenization}]
  We have to show that \eqref{eq:weak:Limit-NonComplete:E-H} is
  satisfied for the right choice of function spaces
  $Y \in \{Y_-, Y_i, Y_{12}\}$ and $X \in \{X_-, X_i, X_{12}\}$.  We
  recall that we have already verified \eqref
  {eq:weak:Limit-NonComplete:E} for $\phi\in X_{12}$ and
  $\psi \in Y_{12}$, see
  Lemma~\ref{lem:Homogenization-away-Interface}.

  \smallskip
  If $\Sigma^\eta_Y$ is asymptotically connecting in the direction
  $e_i$, Proposition~\ref{prop:Trace-Ei} shows that \eqref
  {eq:weak:Limit-NonComplete:E} is satisfied in $\Omega_\pm$ for
  $\psi\in C^1(\overline{\Omega_\pm}, e_j\, \C)$.  Together with the fact
  that the equation holds for all $\psi \in Y_{12}$, this implies that
  \eqref {eq:weak:Limit-NonComplete:E} holds for all $\psi\in Y_i$.

  Using this observation for $i= 1$ and $i=2$, we obtain the claim of
  the theorem in Case 1, where the structure is connected in both
  directions and the space of test-functions is $Y_- = Y_1 \cup Y_2$.

  \smallskip When $\Sigma^\eta_Y$ is asymptotically disconnected in
  the direction $e_j$, then Proposition~\ref{prop:Trace-Hi} shows that
  \eqref{eq:weak:Limit-NonComplete:H} holds for
  $\phi \in C^1_0(\overline{\Omega}, e_j \C)$.  Together with the fact
  that the equation holds for all $\phi \in X_{12}$, this implies that
  \eqref {eq:weak:Limit-NonComplete:E} holds for all $\phi\in X_i$.
  This implies the theorem in Case 2 where the structure is
  disconnected in both directions and the space of test-functions is
  $X_- = X_1 \cup X_2$.

  \smallskip Finally, let us consider Case 3, $\Sigma^\eta_Y$ is
  asymptotically connecting in the direction $e_i$ and asymptotically
  disconnected in direction $e_j$. As shown above, this implies that
  test-functions $\psi\in Y_i$ and $\phi\in X_i$ are permitted. These
  are the desired spaces of test-functions in Case 3.
\end{proof}

\section{Connectedness properties of general geometric structures}
\label{sec:connected-general}

The last section concluded the homogenization result. We now change
our perspective and turn to the analysis of connectedness properties
of geometric structures. In this section, we conclude connectedness
properties of one geometric structure from the connectedness
properties of another geometric structure.

The observations of this section are very valuable when they are
combined with our results on wires in Section \ref{sec:wires}.  The
monotonicity property allows to conclude: Every structure that
contains $e_i$-connected wires is also $e_i$-connected. The
deformation argument allows to conclude: Every smoothly deformed
$e_i$-connected wire is also $e_i$-connected.

\subsection{Monotonicity property for the obstacles}
\label{ssec:MonotonicityArgument}
We show that obstacles $\Sigma^\eta_Y$ are asymptotically connecting
if we can find subsets that are asymptotically connecting.  Vice
versa, obstacles $\widetilde\Sigma^\eta_Y$ are asymptotically disconnected
if we can find supersets that are asymptotically disconnected.
\begin{proposition}[Monotonicity]
  \label{prop:GeometricMonotonicity}
  Let $i \in \{1,2\}$ be a direction, let $\Sigma^\eta_Y$ and
  $\widetilde{\Sigma}^\eta_Y$ be two sequences of obstacles that are
  ordered in the sense of an inclusion:
  $\widetilde{\Sigma}^\eta_Y \subset \Sigma^\eta_Y$. Then, there holds:

  (i) When $\widetilde{\Sigma}^\eta_Y$ is asymptotically connecting in
  direction $i$, then also $\Sigma^\eta_Y$ is.

  (ii) When $\Sigma^\eta_Y$ is asymptotically disconnected in
  direction $i$, then also $\widetilde{\Sigma}^\eta_Y$ is.
\end{proposition}

\begin{proof}
  We consider $\widetilde{\Sigma}^\eta_Y$ that is asymptotically
  connecting. By definition of this property, there exists a sequence
  of cell functions $\Psi^{(i)}_\eta \in \Hd(\curl, Z)$ that satisfies
  \eqref{eq:CellProblem:Psi} for an integration domain
  $\widetilde{Z}^\eta \coloneqq Z \setminus \widetilde{\Sigma}^\eta_Y$. The
  same sequence of cell functions $\Psi^{(i)}_\eta \in \Hd(\curl, Z)$
  satisfies \eqref{eq:CellProblem:Psi} for the integration
  domain $Z^\eta$, since this domain is smaller:
  $Z^\eta = Z \setminus \Sigma^\eta_Y \subset Z \setminus
  \widetilde{\Sigma}^\eta_Y = \widetilde{Z}^\eta$.  We conclude that
  $\Sigma^\eta_Y$ is asymptotically connecting.
	
  \smallskip Let $\Sigma^\eta_Y$ be asymptotically disconnected. Then,
  there exists a sequence of cell functions
  $\Phi^{(i)}_\eta \in \Hd(\curl, Z)$ that satisfies
  \eqref{eq:CellProblem:Phi} and vanishes on
  $\Sigma^\eta_Y$. In particular, $\Phi^{(i)}_\eta \in \Hd(\curl, Z)$
  vanishes on $\widetilde{\Sigma}^\eta_Y\subset \Sigma^\eta_Y$. The same
  sequence can be used for $\widetilde{\Sigma}^\eta_Y$ and we conclude
  that $\widetilde{\Sigma}^\eta_Y$ is asymptotically disconnected.
\end{proof}

\subsection{General deformation argument}
\label{ssec:DeformationArgument}

The following transformation argument allows to extend the results for
straight geometries to deformed geometries.

\begin{proposition}[Deformation of inclusions]
  \label{prop:construction:Psi-wires-perfect:curved}
  Let $i \in \{1,2\}$ and $\widehat{\Sigma}_Y^\eta$ be a sequence of
  obstacles that is asymptotically connecting in the sense of
  Definition \ref{def:Analytical-Connectedness} (resp.~disconnected in
  the sense of Definition \ref{def:Analytical-DisConnectedness}) in
  the direction $e_i$. Let the sequence of deformation maps
  $\varphi_\eta \colon Y \to Y$ be bi-Lipschitz with
  $\varphi_\eta = \id$ on $\{ x_3 = 1\}$ and on $\{ x_3 = 0\}$, and
  such that $\varphi_\eta - \id$ is $e_1$- and $e_2$-periodic. We
  assume for a constant $C>0$ the uniform boundedness
  \begin{equation}
    \label{eq:est:diff:nabla}
    \|\nabla \varphi_\eta\|_{L^\infty(\R^n)}
    + \|\nabla \varphi_\eta^{-1}\|_{L^\infty(\R^n)} \leq C\,.
  \end{equation}
  Then, $\Sigma_Y^\eta \coloneqq \varphi_\eta(\widehat{\Sigma}_Y^\eta)$ is
  a sequence of obstacles that is asymptotically connecting in the
  sense of Definition \ref{def:Analytical-Connectedness}
  (resp.~disconnected in the sense of Definition
  \ref{def:Analytical-DisConnectedness}) in the direction $e_i$.
\end{proposition}

\begin{proof}
  We present a proof for the case of asymptotically connectedness. For
  asymptotically disconnected obstacles the calculations are
  analogous.  When $\widehat{\Sigma}_Y^\eta$ is asymptotically connected
  in the direction $e_i$, there exists a sequence
  $\widehat{\Psi}^{(i)}_\eta \in \Hd(\curl, Z)$ that satisfies
  \eqref{eq:CellProblem:Psi} for the domain
  $\widehat{Z}^\eta = Z \setminus \widehat{\Sigma}_\eta$.
	
  We use the covariant Piola transform and define $\Psi^{(i)}_\eta$ by
  \begin{align*}
    \Psi^{(i)}_\eta(\varphi_\eta(x))
    = D \varphi_\eta^{-\top}(x) \widehat{\Psi}^{(i)}_\eta(x)\,,
  \end{align*}
  where $D \varphi_\eta$ denotes the Jacobian matrix,
  i.e.~$(D\varphi_\eta)_{l,m} = \partial_{x_m} (\varphi_\eta)_l$. The
  transformation is such that the curl of the new field depends only
  on the curl of the given field (and not on other derivatives):
  \begin{align}\label{trafo:curl-piola}
    (\rot \Psi^{(i)}_\eta) \circ \varphi_\eta
    = \det(D \varphi_\eta)^{-1} D \varphi_\eta\, \rot \widehat{\Psi}^{(i)}_\eta\,,
  \end{align} see~\cite[Corollary~3.58]{Monk2003a}.
  Moreover, since $\varphi_\eta -\id$ is $e_1$- and $e_2$-periodic,
  this periodicity is transferred from $\widehat{\Psi}^{(i)}_\eta$ to $\Psi^{(i)}_\eta$
  and there holds $\Psi^{(i)}_\eta \in \Hd(\rot,Z)$.

  In order to check \eqref{eq:CellProblem:Psi:L2} for
  $\Psi_\eta^{(i)}$, we have to verify, for $j=3-i$, that
  $\eta \|\Psi_\eta^{(i)} - e_j \one{\{x_3 >0\}}\|_{L^2(Z^\eta)}^2$
  vanishes in the limit $\eta\to 0$.  Since $\Psi_\eta^{(i)}$ and
  $\widehat{\Psi}^{(i)}_\eta$ coincide outside $Y$, we only have to show
  $\eta \|\Psi_\eta^{(i)} - e_j \|_{L^2(Y\setminus \Sigma_Y^\eta)}^2
  \to 0$.  The uniform essential boundedness of $D \varphi_\eta$ and
  $(D \varphi_\eta)^{-1}$ implies
  $\|\det(\nabla \varphi_\eta)^{-1}\|_{L^\infty(\R^n)} \leq C$ and we
  can calculate
  \begin{align*}
    \eta \|\Psi_\eta^{(i)}
    &- e_j\|_{L^2(Y\setminus \Sigma_Y^\eta)}^2 
      = \eta \int\limits_{Y\setminus \Sigma_Y^\eta} |\Psi_\eta^{(i)} - e_j|^2 
    \\
    &= \eta\int\limits_{Y\setminus \varphi_\eta^{-1}(\Sigma_Y^\eta)}
      |\det(D \varphi_\eta)|\, |\Psi_\eta^{(i)} \circ \varphi_\eta - e_j |^2
    \displaybreak[2]\\
    &\le
      \eta\int\limits_{Y\setminus \widehat{\Sigma}_Y^\eta} |\det(D \varphi_\eta)|\,
      | D \varphi_\eta^{-\top} \widehat{\Psi}_\eta^{(i)} |^2 + C\, \eta
      \displaybreak[2]\\
    & \le C\, \eta \int\limits_{Y\setminus \widehat{\Sigma}_Y^\eta} 
      | \widehat{\Psi}_\eta^{(i)} |^2 + C\, \eta
      \le C\, \eta \int\limits_{Y\setminus \widehat{\Sigma}_Y^\eta} 
      | \widehat{\Psi}_\eta^{(i)} - e_j |^2 + C\, \eta \to 0\,,
  \end{align*}
  where we allow that $C$ changes from one inequality to the next. In
  the convergence we used that $\widehat{\Psi}_\eta^{(i)}$ satisfies
  \eqref{eq:CellProblem:Psi:L2}.

  To show that $\Psi_\eta^{(i)}$ satisfies
  \eqref{eq:CellProblem:Psi:curl}, we use \eqref{trafo:curl-piola} and
  calculate, on the relevant domain $Y\setminus \Sigma_Y^\eta$,
  \begin{align*}
    &\eta^{-1} \|\curl \Psi_\eta^{(i)}\|_{L^2(Y\setminus \Sigma_Y^\eta)}^2
      = \eta^{-1} \int\limits_{Y\setminus \Sigma_Y^\eta} |\curl \Psi_\eta^{(i)}|^2 \dx\\
    &\qquad =
      \eta^{-1} \int\limits_{Y\setminus \widehat\Sigma_Y^\eta} |\det(D \varphi_\eta)|
      |(\curl \Psi_\eta^{(i)}) \circ \varphi_\eta|^2
      \displaybreak[2]\\
    &\qquad =	\eta^{-1} \int\limits_{Y\setminus \widehat\Sigma_Y^\eta} |\det(D \varphi_\eta)|
      |\det(D \varphi_\eta)^{-1} D \varphi_\eta \rot \widehat{\Psi}^{(i)}_\eta|^2
      \displaybreak[2]\\
    &\qquad \le C\, \eta^{-1} \int\limits_{Y\setminus \widehat\Sigma_Y^\eta}
      | \rot \widehat{\Psi}^{(i)}_\eta|^2 \to 0\,,
  \end{align*}
  where we used in the convergence that that $\widehat{\Psi}_\eta^{(i)}$
  satisfies \eqref{eq:CellProblem:Psi:curl}.  This completes the
  verification of \eqref{eq:CellProblem:Psi}. As already indicated,
  the calculations for $\Phi_\eta^{(i)}$ are analogous, the proof is
  complete.
\end{proof}

\section{Wire geometries}
\label{sec:wires}

In this section we analyze inclusions that model wires. We start from
a concrete set $\Sigma^\eta_Y$ that stands for a segment of a single
wire. Putting together these pieces to the set $\Sigma_\eta$, we
obtain $O(\eta^{-1})$ wires that are periodically distributed at a
distance $\eta$, they are parallel, all oriented in direction $e_1$.

Our result is that if the wires are not too thin, these inclusions are
indeed connecting in direction $e_1$ and not connecting in direction
$e_2$.  The situation becomes interesting when the wire (more
precisely, its generator $\Sigma^\eta_Y$ in the periodicity cell $Y$)
has a radius that tends to $0$ as $\eta\to 0$. We find the critical
scaling of the radius such that the wire remains connecting in
direction $e_1$. We also treat a situation in which the wire is
interrupted by thin gaps. We check that very thin gaps do not destroy
the connectedness property in direction $e_1$. This shows that,
indeed, our asymptotic connectedness concept is different from the
classical connectedness of sets.

We recall that asymptotically connecting in direction $i = 1$ is
defined by the existence of cell functions $\Psi^{(i)}_\eta$, see
Definition~\ref{def:Analytical-Connectedness}.  Being asymptotically
disconnected in direction $j = 2$ is declared by the existence of cell
functions $\Phi^{(j)}_\eta$, see
Definition~\ref{def:Analytical-DisConnectedness}.  Thus, we are
interested in constructing functions $\Phi_\eta^{(1)}$ and
$\Psi_\eta^{(2)}$; in this section we skip the superscript and write
only $\Phi_\eta$ and $\Psi_\eta$.  We prepare the results with
Section~\ref{ssec:DiffCalc2D}, where we construct $2$D auxiliary
functions. In Section~\ref{ssec:wires}, the $3$D cell functions are
constructed from their $2$D counterparts.

\subsection{Constructions in \texorpdfstring{$2$}{2}D}
\label{ssec:DiffCalc2D}
We consider inclusions (wires), that are not varying in the direction
$e_1$. In such a geometry, it is convenient to reduce the construction
of the cell functions to a $2$D problem. We use the following notation
in two space dimensions: For a vector $a = (a_1,a_2) \in \R^2$, we
define the rotated vector by $a^\perp \coloneqq (-a_2,a_1)$. The
vector $a^\perp$ is obtained by applying the rotation matrix: 
\begin{equation*}
  R \coloneqq \left(
    \begin{array}{cc}
      0 & -1\\
      1 & 0
    \end{array}
  \right)\,,\qquad a^\perp = R a\,.
\end{equation*}
When $U \subset \R^2$ is a Lipschitz domain and $n$ is the unit outer
normal, the left tangential vector is $t \coloneqq n^\perp$.  For a
function $u \colon \R^2 \to \R$, we define the rotated gradient by
$\nabla^\perp u \coloneqq (-\partial_{2} u, \partial_{1} u)$. For a vector
field $v \colon \R^2 \to \R^2$, we define the two-dimensional curl of
$v$ by
$\nabla^\perp \cdot v \coloneqq (-\partial_{2} v_1 + \partial_{1} v_2)
= -\div\left(R v\right)$. There holds
$\nabla^\perp \cdot \nabla^\perp u = \Delta u$.

The notion of a weak rotated gradient or weak rotated divergence is
defined in the distributional sense.

\textit{Connection to $3D$ calculus:} For a set $U \subset \R^2$, we
consider the corresponding cylindrical domain in three dimensions,
$\widehat{U} \coloneqq \{x \in \R^3 \mid x_1\in \R\,,\, (x_2, x_3) \in U\}$. When
$v \colon U \to \R^2$ is a two-dimensional vector field we define a
three-dimensional vector field $\hat{v} \colon \widehat U \to \R^3$ by
\begin{align}\label{eq:V=v}
  \hat{v}(x_1, x_2, x_3) = e_2\, v_1(x_2,x_3) + e_3\, v_2(x_2,x_3)\,.
\end{align}
The curl of the new function is
\begin{align*}
  \curl \hat{v} (x_1, x_2, x_3) = e_1\, (\nabla^\perp \cdot v)(x_2, x_3)\,.
\end{align*}

The above observations motivate the following construction: We start
from a harmonic potential $u$ in $2$D and construct a two-dimensional
field $v$ as the rotated gradient of $u$, i.e.~$v = \nabla^\perp u$.
Then, the two-dimensional curl of $v$ vanishes:
$\nabla^\perp\cdot v = \nabla^\perp\cdot \nabla^\perp u = \Delta u =
0$. In particular, for such a two-dimensional field $v$, the
three-dimensional curl of $\hat{v}$ vanishes, $\curl \hat{v} = 0$.
Following this idea, we will construct cell-functions as rotated
gradients.

\subsubsection*{Construction of the \texorpdfstring{$2$}{2}D cell
  function \texorpdfstring{$\psi$}{psi}}

Throughout this section, the two-dimensional geometry is given by a
periodicity cell $V \coloneqq (0,1)^2$, the independent variables are
$z = (z_1,z_2)\in V$, given is a point $z_0\in V$ and a radius $r>0$
such that $\overline{B_r(z_0)}\subset V$. In the following, we
suppress the dependence on $z_0$, but we indicate the dependence on
$r>0$. Later on, $r$ will be the radius of the wire, and the
asymptotics of the wire geometry will be relevant in our analysis.

The main part of our analysis regards the verification of the fact
that the wire is connected in direction $e_1$; this requires the
construction of a function $\Psi_\eta$. The three-dimensional function
$\Psi_\eta$ will be obtained from a two-dimensional function $\psi_r$
as sketched above. The two-dimensional function is constructed as
\begin{align}
  \label{eq:psi=rot}
  \psi_r(z) \coloneqq \begin{cases}
    \nabla^\perp \urs(z) &\text{for } z \in B_r(z_0) \,,
    \\
    \nabla^\perp \vrs(z) &\text{for } z \in V\setminus B_r(z_0) \eqqcolon V_r\,.
  \end{cases}
\end{align}
The potentials $\vrs$ and $\urs$ will be obtained as solutions of
Poisson problems.  We note that potentials may have jumps across
$\del B_r(z_0)$, but we have to make sure that the tangential
component of $\psi_r$ has no jumps. We therefore have to demand that
the normal derivatives of the potentials have no jump at
$\partial B_r(z_0)$. Additionally, $\urs$ must be $z_1$-periodic.  We
construct $\vrs$ in Lemma~\ref{lem:vr-psi-wires} and $\urs$ in
Lemma~\ref{lem:ur-psi-wires}.

\begin{lemma}[Potential $\vrs$]\label{lem:vr-psi-wires}
  Let $z_0 \in V = (0,1)^2$ be a point, $R>0$ be a radius such that
  $\overline{B_R(z_0)} \subset V$. We consider $0<r\leq R$ and
  $V_r = V \setminus B_r(z_0)$. There exists a unique solution
  $\vrs \in H^1(V_r)$, periodic in $e_1$-direction, of the problem
  \begin{subequations}\label{eq:CellProblem:VectorPotential:u}
    \begin{align}\label{eq:CellProblem:VectorPotential:v:Laplace}
      -\Delta \vrs &= 0 &&\textrm{in } V_r \,,
      \\\label{eq:CellProblem:VectorPotential:v:Boundary1}
      \partial_2 \vrs &= -1 &&\textrm{on } \{z_2 = 1\} \,, 
      \\\label{eq:CellProblem:VectorPotential:v:Boundary2}
      \partial_2 \vrs &= 0 &&\textrm{on }\{z_2 = 0\} \,,
      \\\label{eq:CellProblem:VectorPotential:v:Boundary3}
      \nabla \vrs \cdot n &= \frac{1}{|\partial B_r(z_0)|}
                        &&\textrm{on } \partial B_r(z_0) \,,
    \end{align}
    with the normalization
    \begin{align}\label{eq:CellProblem:VectorPotential:v:MeanValue}
      \int\limits_{\partial B_{R}(z_0)} \vrs &= 0 \, .
    \end{align}
  \end{subequations}
  Moreover, there exists a constant $C_1 = C_1(R) > 0$, independent of
  $r$, such that
  \begin{align}\label{eq:VectorPotential:v:estimate}
    \| \nabla \vrs \|^2_{L^2(V_r)} \leq C_1 | \ln (r) | \,.
  \end{align}
\end{lemma}

\begin{proof}
  We note that \eqref{eq:CellProblem:VectorPotential:v:Laplace} is
  compatible with the boundary conditions \eqref
  {eq:CellProblem:VectorPotential:v:Boundary1}--\eqref
  {eq:CellProblem:VectorPotential:v:Boundary3}. We denote by $X_r$ the
  subset of $H^1(V_r)$ consisting of $e_1$-periodic functions
  satisfying \eqref{eq:CellProblem:VectorPotential:v:MeanValue}. We
  apply the theorem of Lax--Milgram and obtain the existence and
  uniqueness of a solution $\vrs \in X_r$ of
  \eqref{eq:CellProblem:VectorPotential:u}.
	
  To derive estimate \eqref{eq:VectorPotential:v:estimate}, we
  reformulate \eqref{eq:CellProblem:VectorPotential:u} as a
  minimization problem. Let the energy $E \colon X_r \to \R$ be given
  by
  \begin{align*}
    E(u) \coloneqq \frac{1}{2} \int\limits_{V_r}
    |\nabla u|^2 - \frac{1}{\partial B_r(z_0)}\int\limits_{\partial B_r(z_0)} u
    + \int\limits_{\{z_2 = 1 \}} u \, .
  \end{align*}
  We choose $u = 0 \in X_r$ as a competitor and obtain
  $E(\vrs) \leq E(0) = 0$. Hence, we can estimate
  \begin{align}\label{eq:EstE(vrs)}
    \frac{1}{2} \int\limits_{V_r} |\nabla \vrs|^2 \leq \frac{1}{\partial B_r(z_0)}\int\limits_{\partial B_r(z_0)} \vrs - \int\limits_{\{z_2 = 1 \}} \vrs \, .
  \end{align}
  To derive \eqref{eq:VectorPotential:v:estimate}, it therefore
  suffices to estimate the two integrals on the right-hand side of
  \eqref{eq:EstE(vrs)}.
	
  For radii $r \leq s \leq R$ and, we apply the theorem of Gauß on the
  domain $B_R(z_0) \setminus \overline{B_s(z_0)}$ and find
  \begin{align*}
    \int\limits_{\partial B_s (z_0)} n \cdot \nabla \vrs = -1 \,,
  \end{align*}
  where $n$ denotes the outer normal of $B_s(z_0)$.  We rewrite this
  equality as
  \begin{align*}
    -\frac{1}{2 \pi s}
    &= \frac{1}{|\partial B_s (z_0)|} \int\limits_{\partial B_s (z_0)} n \cdot \nabla \vrs
      = \frac{1}{|\partial B_1 (z_0)|} \int\limits_{\partial B_1 (0)}
      \zeta \cdot \nabla \vrs(z_0 + s \zeta) \dd S(\zeta) 
    \\
    &=\frac{1}{2 \pi}\partial_s \int\limits_{\partial B_1(0)} \vrs(z_0 + s \zeta) \dd S(\zeta) 
      =\partial_s \Bigg( \frac{1}{|\partial B_s(z_0)|}
      \int\limits_{\partial B_s(z_0)} \vrs \Bigg)\,.
  \end{align*}
  Because of \eqref{eq:CellProblem:VectorPotential:v:MeanValue}, we
  obtain from the fundamental theorem of calculus
  \begin{align}
    \begin{aligned}\label{eq:est:E:inner-Boundary}
      \frac{1}{|\partial B_r(z_0)|} \int\limits_{\partial B_r(z_0)}
      \vrs &=-\int\limits_r^R\partial_s \Bigg( \frac{1}{|\partial
        B_s(z_0)|}
      \int\limits_{\partial B_s(z_0)} \vrs \Bigg) \dd s \\
      &= \int\limits_r^R \frac{1}{2 \pi s} \dd s = \frac{1}{2\pi}
      (\ln(R) - \ln(r)) \,.
    \end{aligned}
  \end{align}
  This provides a bound for the second integral of the energy $E(u)$ for
  $u = \vrs$.
	
  The last integral of $E$ can be estimated with the trace theorem on
  $V_R$, with a Poincar\'e inequality and the Young inequality as
  \begin{align}\label{eq:est:E:outer-Boundary}
    -\int\limits_{\{z_2 = 1 \}} \vrs \leq C_T \|\vrs\|_{H^1(V_R)}
    \leq C_P \|\nabla \vrs\|_{L^2(V_R)} \leq C_Y + \frac{1}{4} \|\nabla \vrs\|_{L^2(V_R)}^2 \,.
  \end{align}
  Inserting \eqref{eq:est:E:inner-Boundary} and
  \eqref{eq:est:E:outer-Boundary} in \eqref{eq:EstE(vrs)} yields
  \begin{align*}
    \frac{1}{4} \|\nabla \vrs\|_{L^2(V_r)}^2
    \leq \frac{1}{2\pi} (\ln(R) - \ln(r)) + C_Y \leq 4 C_1  |\ln (r)| \,,
  \end{align*} 
  where the constant $C_1$ depends on $R$.  This shows
  \eqref{eq:VectorPotential:v:estimate}.
\end{proof}

The following lemma provides the scalar potential $\urs$ on $B_r(z_0)$.
\begin{lemma}[Potential $\urs$]\label{lem:ur-psi-wires}
  Let $r>0$ and $z_0 \in \R^2$.  There exists a solution
  $\urs \in H^1(B_r(z_0))$ of the problem
  \begin{subequations}\label{eq:CellProblem:VectorPotential:u-inside-wire}
    \begin{align}\label{eq:CellProblem:VectorPotential:u-inside-wire:1}
      -\Delta \urs &= \frac{1}{|B_r(z_0)|} &&\textrm{in } B_r(z_0) \,,
      \\\label{eq:CellProblem:VectorPotential:u-inside-wire:2}
      \nabla \urs \cdot n &= -\frac{1}{|\partial B_r(z_0)|} &&\textrm{on } \partial B_r(z_0) 
    \end{align}
  \end{subequations}
  and constants $C_2$ and $C_3$ such that
  \begin{align}\label{eq:VectorPotential:u:estimate}
    \|\nabla \urs \|_{L^2(B_r(z_0))}^2 = C_2 \,, \qquad \|\Delta \urs\|_{L^2(B_r(z_0))}^2 = C_3 r^{-2}\,.
  \end{align}
\end{lemma}

\begin{proof}
  We use $v_1(z_1,z_2) \coloneqq -(4 \pi)^{-1} (z_1^2 + z_2^2)$ and
  $\urs(z) \coloneqq v_1((z-z_0)/r)$ for $r >0$.  An explicit
  calculation shows that $-\Delta v_1 = \pi^{-1}$ in $B_1(0)$ and
  $\nabla v_1 \cdot n = -(2\pi)^{-1}$ on $\partial B_1(0)$. With the
  chain rule, we obtain $-\Delta \urs =\pi^{-1} r^{-2}$ in $B_r(0)$
  and $\nabla \urs \cdot n = -(2\pi)^{-1} r^{-1}$ on
  $\partial B_r(0)$. Thus, $\urs$ solves
  \eqref{eq:CellProblem:VectorPotential:u-inside-wire}. In the same
  manner, the scaling by $r$ leads to
  \eqref{eq:VectorPotential:u:estimate}, an explicit computation shows
  $C_2 = (8 \pi)^{-1}$ and $C_3 = \pi^{-1}$.
\end{proof}

We combine the $2$D potentials $\vrs$ and $\urs$ to define the $2$D
cell functions $\psi_r$ via \eqref{eq:psi=rot}. Let us calculate the
norms of the relevant quantities.

\begin{lemma}[$2$D vector field $\psi_r$]
  \label{lem:psi-wires}
  Let $z_0 \in V = (0,1)^2$ be a point, $R>0$ be a radius such that
  $\overline{B_R(z_0)} \subset V$. We consider $0<r< R$ and
  $V_r = V \setminus B_r(z_0)$. There exists $\psi_r \in L^2(V, \R^2)$
  with $\nabla^\perp \cdot\psi_r \in L^2(V, \R^2)$, periodic in the
  sense that its $e_1$-periodic extension $\widetilde\psi_r$ satisfies
  $\nabla^\perp\cdot \widetilde\psi_r \in L^2_\loc(\R\times (0,1))$, with
  the upper and lower boundary conditions
  \begin{equation}
    \label{eq:boundary-cond-psi-r}
    \psi_r|_{\{z_2 = 0\}} \cdot e_1 = 0 \,,\qquad \qquad 
    \psi_r|_{\{z_2 = 1\}} \cdot e_1 = 1\,.
  \end{equation}
  Furthermore, for constants $C_1, C_2, C_3$:
  \begin{align}
    \label{eq:est:psi_r}
    \|\psi_r\|_{L^2(V_r)}^2 &\leq C_1 |\ln(r)| \,,
    &
      \|\psi_r\|_{L^2(B_r(z_0))}^2 &= C_2  \,,
    \\\label{eq:est:psi_r:curl}
    \|\nabla^\perp \cdot \psi_r\|_{L^2(V_r)}^2 &= 0 \,,
    &
      \|\nabla^\perp \cdot \psi_r\|_{L^2(B_r(z_0))}^2 &= C_3 r^{-2}  \,.
  \end{align}
  The constants $C_1, C_2, C_3$ depend on $R$, but are independent of
  $r$, they are the constants of Lemmas~\ref{lem:vr-psi-wires} and
  \ref{lem:ur-psi-wires}.
\end{lemma}
  
\begin{proof}
  Let $\vrs$ and $\urs$ be given by Lemma~\ref{lem:vr-psi-wires} and
  Lemma~\ref{lem:ur-psi-wires}, respectively. We define
  $\psi_r \coloneqq \one{B_r(z_0)}\nabla^\perp \urs +
  \one{V_r}\nabla^\perp \vrs$.

  The first estimate of \eqref{eq:est:psi_r} is a direct consequence
  of the estimate \eqref{eq:VectorPotential:v:estimate} for
  $\nabla \vrs$, see Lemma~\ref{lem:vr-psi-wires}. The second estimate
  of \eqref{eq:est:psi_r} is a direct consequence of
  \eqref{eq:VectorPotential:u:estimate} in Lemma
  \ref{lem:ur-psi-wires}.  Since $\Delta \vrs = 0$, we get the first
  equality in \eqref{eq:est:psi_r:curl}. The second
  equality follows from
  $\|\Delta \urs\|_{L^2(B_r(z_0))}^2 = C_3 r^{-2}$ of \eqref
  {eq:est:psi_r:curl}.

  The boundary conditions \eqref{eq:boundary-cond-psi-r} follow from
  \eqref{eq:CellProblem:VectorPotential:v:Boundary1} and
  \eqref{eq:CellProblem:VectorPotential:v:Boundary2}.

  Regarding $e_1$-periodicity of $\psi_r$, we recall that $\vrs$ is
  the periodic solution of a Poisson problem. This does not only imply
  $\vrs(1,z_2) = \vrs(0,z_2)$ for every $z_2\in (0,1)$, but also
  $\del_1 \vrs(1,z_2) = \del_1 \vrs(0,z_2)$ for every $z_2\in (0,1)$,
  both conditions are understood in the sense of traces.  This shows
  that $\psi_r$ is $e_1$-periodic in the sense as described in the
  lemma.
  
  We finally note that, by \eqref
  {eq:CellProblem:VectorPotential:v:Boundary3} and \eqref
  {eq:CellProblem:VectorPotential:u-inside-wire:2} the normal
  components of $\nabla \vrs$ and $ \nabla \urs$ coincide at
  $\partial B_r(z_0)$. This implies that tangential components of
  $\psi_r$ do not jump at $\partial B_r(z_0)$, hence,
  $\nabla^\perp \cdot\psi_r \in L^2(V, \R^2)$.
\end{proof}

\subsubsection*{Construction of the \texorpdfstring{$2$}{2}D cell
  function \texorpdfstring{$\phi$}{phi}}

The construction of $\phi$ in two dimensions is much simpler than the
construction of $\psi_r$, since it is independent of the radius. In
the subsequent lemma, the set $U \subset (0,1)^2$ is contained in the
complement of the cross section of the conductor, one may think of
$U = V \setminus \overline{B_R(z_0)}$.

\begin{lemma}[$2$D vector field $\phi$]\label{lem:phi-wires}
  Let $U \subset V \coloneqq (0,1)^2$ be a connected Lipschitz domain
  such that $\partial U \supset (0,1) \times \{0,1\}$. There exists
  $\phi \in L^2(V)$ with $\nabla^\perp \cdot\phi \in L^2(V)$,
  $e_1$-periodic in the sense that its $e_1$-periodic extension
  $\widetilde\phi$ satisfies
  $\nabla^\perp\cdot \widetilde\phi \in L^2_\loc(\R\times (0,1))$, satisfying
  \begin{align}\label{eq:boundary-cond-phi:oben-unten}	
    \phi|_{\{z_2 = 0\}} \cdot e_1 &= 1\,,\qquad \qquad   \phi|_{\{z_2 = 1\}} \cdot e_1 = 1\,,\\
	\label{eq:phi:innen}
    \phi|_{\{V \setminus U\}} &= 0\,,\\
    \label{eq:phi:curl}	
    \nabla^\perp \cdot \phi  &= 0 \quad \text{ in } V\,.
  \end{align}
\end{lemma}

\begin{proof}
  We consider the system
  \begin{subequations}\label{eq:CellProblem:VectorPotential:u-phi}
    \begin{align}\label{eq:CellProblem:VectorPotential:u-phi:Laplace}
      -\Delta u^\phi &= 0 &&\textrm{in } U \,,
      \\\label{eq:CellProblem:VectorPotential:u-phi:Boundary1}
      e_2\cdot\nabla u^\phi &= -1 &&\textrm{on } \{z_2 = 0\}\cup \{z_2 = 1\} \,,
      \\\label{eq:CellProblem:VectorPotential:u-phi:Boundary2}
      n\cdot \nabla u^\phi &= 0 &&\textrm{on } \partial U \setminus \partial V \,,
    \end{align}
  \end{subequations}
  with periodicity conditions in $e_1$-direction.  The Lemma of
  Lax--Milgram provides a unique (up to additive constants) solution
  $u^\phi \in H^1(U)$ of \eqref{eq:CellProblem:VectorPotential:u-phi}.

  We define $\phi = \one{U}\nabla^\perp u^\phi$ on $V$.  There holds
  $\nabla^\perp \cdot \phi =\one{U} \Delta u^\phi = 0$ in $U$. Due to
  \eqref{eq:CellProblem:VectorPotential:u-phi:Boundary2}, the
  tangential component of $\phi$ vanishes at
  $\partial U \setminus \partial V$. This implies \eqref {eq:phi:curl}
  in all of $V$.  Condition \eqref{eq:boundary-cond-phi:oben-unten} on
  the boundary values follows from
  \eqref{eq:CellProblem:VectorPotential:u-phi:Boundary1}. The $e_1$
  periodicity of $u^\phi$ is inherited by $\phi$.
\end{proof}

\subsection{Construction of cell functions for
  \texorpdfstring{$3$}{3}D wires}\label{ssec:wires}

In this section, we use the $2$D functions of the last subsection to
construct the $3$D cell functions $\Phi^{(j)}_\eta$ and
$\Psi^{(i)}_\eta$ for wires. These cell functions show that the wire
is asymptotically connecting in the direction $e_i$ and asymptotically
disconnected in direction $e_j$ for $j = 3-i$.

To simplify notation, we consider only $i=1$. Wires are defined with a
radius $r_\eta$ and a center $z_0 \in V= (0,1)^2$ in the two
dimensional cross section.  With the two-dimensional ball
$B_{r_\eta}(z_0) \subset V$, we define the wire as
\begin{equation}
  \label{def:T_r}
  T_{r_\eta} \coloneqq (0,1) \times B_{r_\eta}(z_0)\,.
\end{equation}
The asymptotic connectedness of the wire $T_{r_\eta}$ will depend on
the asymptotic behavior of the radius $r_\eta$.

Additionally, we want to study the effect of gaps in the wire. Let
$I_\eta \subset(0,1)$ be an open set of finitely many intervals
representing gaps in the wires. We define
\begin{align}\label{def:T_I,r}
  &T_{r_\eta,I_\eta } \coloneqq((0,1)\setminus I_\eta) \times
    B_{r_\eta}(z_0)\,.
\end{align}
Note that $T_{r_\eta,I_\eta }$ is also defined for
$I_\eta = \emptyset$; in this case, no gaps are studied,
$T_{r_\eta} = T_{r_\eta, \emptyset}$.  The geometries are visualized
in Figure~\ref{fig:WireGeoemetries}.

\begin{figure}[ht]
  \centering
  \begin{subfigure}{.48\textwidth}
    \centering
    \includegraphics[width=0.98\linewidth]{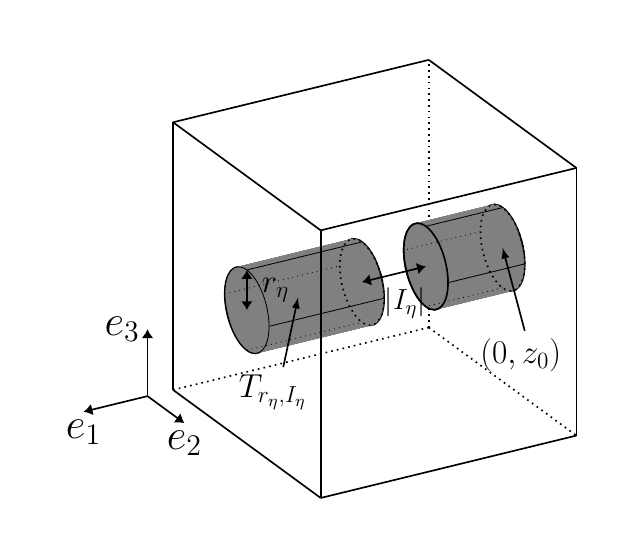} 
    \subcaption{The set $T_{r_\eta, I_\eta}$ in the
      reference cell $Y$}
    \label{fig:Sigma-Slit-756}
  \end{subfigure}
  \hspace{\fill}
  \begin{subfigure}{.48\textwidth}
    \centering
    \includegraphics[width=0.98\columnwidth]{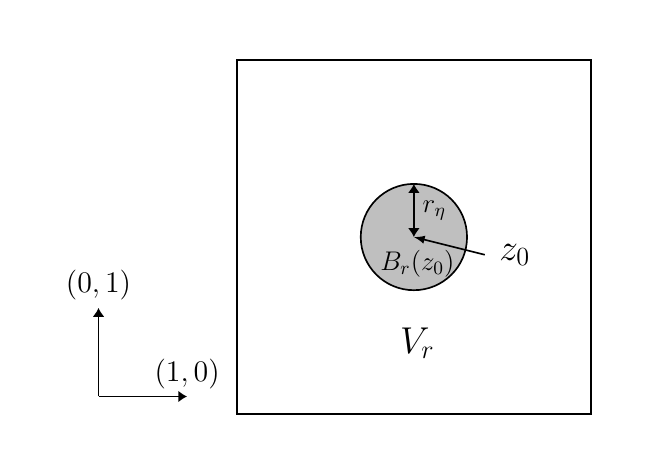}
    \subcaption{The cross section of $T_{r_\eta, I_\eta}$ in
      $(0,1)^2$} \label{fig:Z_eta}
  \end{subfigure}
  \caption{Geometries for the construction of $\phi_\eta$. The vector
    $(1,0) \in \R^2$ in the $2$D picture corresponds with
    $e_2 = (0,1,0)$ in the $3$D picture and $(0,1)\in \R^2$ with
    $e_3 = (0,0,1)\in \R^3$. When performing calculations in the $2$D
    context, we also write $e_1$ for $(1,0) \in \R^2$ and $e_2$ for
    $(0,1)\in \R^2$.}
  \label{fig:WireGeoemetries}
\end{figure}

\smallskip For the wires $T_{r_\eta,I_\eta }$ of \eqref{def:T_I,r} we
obtain the following results:
\begin{itemize}
\item Proposition~\ref{prop:construction:Psi-wires-perfect}: If
  $\eta |\ln(r_\eta)| \to 0$ and
  $\eta^{-1} r_\eta^{-2} |I_\eta| \to 0$, then $T_{r_\eta, I_\eta}$ is
  asymptotically connecting in the direction $e_1$.
\item Proposition~\ref{prop:construction:Phi-wires-ortho}: Wires are
  asymptotically disconnected in direction $e_2$.
\item Proposition~\ref{prop:construction:Phi-wires-thin}: If
  $\eta |\ln(r_\eta)| \to \infty$, then
  $T_{r_\eta}= T_{r_\eta, \emptyset}$ is asymptotically disconnected
  in the direction $e_1$.
\end{itemize}

In the following, we construct the sequences of $3$D cell functions
from their $2$D counterparts, exploiting that, apart from the gaps,
the geometry is independent of $x_1$.  Starting from
$\psi_{r_\eta} \colon V = (0,1)^2 \to \R^2$ given by
Lemma~\ref{lem:psi-wires}, we define $\Psi_\eta$ as
\begin{align}
  \label{eq:Con:Psi-psi}
  &\Psi_\eta(x_1,x_2,x_3) \coloneqq
        \left(\begin{array}{c}
            0 \\
            \psi_{r_\eta, 1}\\
            \psi_{r_\eta, 2}
          \end{array}
  \right) (x_2,x_3)\, 
  \one{\{0< x_3<1\}}
  + e_2\, \one{\{x_3>1\}} \,.
\end{align}
The properties of $\psi_{r_\eta}$ on $\{z_2 = 1\}$ and $\{z_2 = 0\}$
imply that the tangential components of the $3$D functions have no
jumps.  The $\curl$ is therefore given by
\begin{align}
  \label{eq:curl-Psi=div-psi}
  &\curl \Psi_\eta(x_1,x_2,x_3)
    =  \left(\nabla^\perp \cdot \psi_{r}\right)(x_2,x_3)\, e_1\, \one{\{0< x_3<1\}}\,.
\end{align}

\begin{proposition}[Cell function $\Psi_\eta$ and connectedness of
  wires in direction $e_1$]
  \label{prop:construction:Psi-wires-perfect}
  Let $z_0 \in V = (0,1)^2$ be a point, $R>0$ be a radius such that
  $\overline{B_R(z_0)} \subset V$. For a sequence of radii $0 < r_\eta \leq R$
  and a sequence of open sets $I_\eta \subset (0,1)$, we consider
  $T_{r_\eta, I_\eta}$ of \eqref{def:T_I,r} and assume
  \begin{equation}
    \label{eq:r-asymptotics}
    \eta |\ln(r_\eta)| \to 0 \qquad \textrm{ and }
    \qquad \eta^{-1} r_\eta^{-2} |I_\eta| \to 0 \, .
  \end{equation} 
  Then, the inclusion $\Sigma_Y^\eta = T_{r_\eta, I_\eta}$ is
  asymptotically connecting in direction $e_1$, i.e.~there exists a
  sequence $\Psi_\eta \in \Hd(\curl, Z)$ satisfying
  \eqref{eq:CellProblem:Psi} for $j = 2$.
\end{proposition}

\begin{proof}
  Let $\psi_{r_\eta} \in H_\#(\nabla^\perp \cdot, V)$ be given by
  Lemma \ref{lem:psi-wires} and $\Psi_\eta$ by \eqref
  {eq:Con:Psi-psi}.  Since $\psi_{r_\eta}$ is periodic in the first
  argument, $\Psi_\eta$ is $e_2$-periodic (since $\Psi_\eta$ is
  independent of $x_1$, it is also $e_1$-periodic).  The curl of
  $\Psi_\eta$ is essentially given by
  $\nabla^\perp \cdot \psi_{r_\eta}$, see \eqref
  {eq:curl-Psi=div-psi}.  This will allow to derive estimates
  \eqref{eq:CellProblem:Psi:L2} and \eqref{eq:CellProblem:Psi:curl}.

  The estimate \eqref{eq:CellProblem:Psi:L2} compares the cell
  function $\Psi_\eta$ with the characteristic function
  $e_2 \one{\{x_3 > 0\}}$ on the set $Z^\eta$. Since $\Psi_\eta$ is
  identical to $e_2$ for $x_3 > 1$ and vanishes for $x_3 < 0$, the
  only contributions to the $L^2(Z^\eta)$-norm come from
  $Y \setminus T_{r_\eta, I_\eta}$, which we decompose into the two
  sets $Y \setminus T_{r_\eta}$ and
  $T_{r_\eta} \setminus T_{r_\eta, I_\eta} = I_\eta \times
  B_{r_\eta}(z_0)$. With a triangle inequality and Young's inequality,
  we find
  \begin{align*}
    \eta\|\Psi_\eta - e_2 \one{\{x_3>0\}}\|_{L^2(Z^\eta)}^2
    &= \eta \| \Psi_\eta - e_2 \|_{L^2(Y \setminus T_{r_\eta})}^2
      + \eta \|\Psi_\eta - e_2 \|_{L^2(I_\eta \times  B_{r_\eta}(z_0))}^2 
    \\
    &\le 2\eta \|\psi_{r_\eta} \|_{L^2(V_{r_\eta})}^2 
      + 2\eta  |I_\eta| \|\psi_{r_\eta}\|_{L^2( B_{r_\eta}(z_0))}^2
      + 2 \eta
    \\
    &\leq 2\eta\, C_1 |\ln(r_\eta)| + 2\eta |I_\eta| C_2  + 2 \eta
      \to 0 \,,
  \end{align*}
  where we used \eqref{eq:est:psi_r} of Lemma \ref{lem:psi-wires} and
  the asymptotics \eqref{eq:r-asymptotics} in the last line.  This
  shows that $\Psi_\eta$ satisfies \eqref{eq:CellProblem:Psi:L2}.

  To show \eqref{eq:CellProblem:Psi:curl}, we proceed
  similarly and calculate
  \begin{align*}
    \eta^{-1}\|\curl \Psi_\eta \|_{L^2(Z^\eta)}^2 
    &= \eta^{-1} \|\curl \Psi_\eta \|_{L^2(Y \setminus T_{r_\eta})}^2 
      + \eta^{-1} \|\curl \Psi_\eta\|_{L^2(I_\eta \times  B_{r_\eta}(z_0))}^2
    \\
    &= \eta^{-1} \|\nabla^\perp \cdot \psi_{r_\eta} \|_{L^2(V_{r_\eta})}^2
      + \eta^{-1} |I_\eta| \|\nabla^\perp \cdot \psi_{r_\eta}\|_{L^2(B_{r_\eta}(z_0))}^2
    \\
    &\leq \eta^{-1} |I_\eta| C_3 r_\eta^{-2} \to 0\,,
  \end{align*}
  where we used estimate \eqref{eq:est:psi_r:curl} of Lemma
  \ref{lem:psi-wires} and the asymptotics \eqref{eq:r-asymptotics} in
  the last line. This shows that $\Psi_\eta$ satisfies
  \eqref{eq:CellProblem:Psi:curl}.
\end{proof}

The cell function $\Phi_\eta$ is constructed similar as $\Psi_\eta$ in
\eqref{eq:Con:Psi-psi}.  For $\Phi_\eta$, the estimates are actually
much simpler than for $\Psi_\eta$, since we can choose
$\Phi_\eta = \Phi$ independent of $\eta$.  Starting from a
vector-field on $\phi \colon U\to \C^2$ on $U\subset V = (0,1)^2$, we
define
\begin{align}
  \label{eq:Con:Phi-phi}
  &\Phi_\eta(x_1,x_2,x_3) \coloneqq \Phi(x_1,x_2,x_3) =
        \left(\begin{array}{c}
            0 \\
            \phi_{1}\\
            \phi_{2}
          \end{array}
  \right) (x_2,x_3)\, 
  \one{\{0< x_3<1\}}
  + e_2\, \one{\{x_3>1\text{ or } x_3<0\}}\,.
\end{align}
Choosing $\phi$ of Lemma~\ref{lem:phi-wires} for $U = V_R$, we can
exploit the properties of $\phi$ on $\{z_2 = 1\}$ and $\{z_2 =
0\}$. They imply that tangential components of $\Phi_\eta$ have no
jump across $x_3 = 0$ or $x_3 = 1$.  In particular, the curl is given
by
\begin{align}
  \label{eq:curl-Phi=div-phi}
  &\curl \Phi_\eta(x_1,x_2,x_3)
    =  \left(\nabla^\perp \cdot \phi\right)(x_2,x_3)\, e_1\, \one{\{0< x_3<1\}}\,.
\end{align}

\begin{proposition}[Cell function $\Phi_\eta$ and disconnectedness of
  wires in direction $e_2$]
  \label{prop:construction:Phi-wires-ortho}
  Let $z_0 \in V = (0,1)^2$ be a point, $R>0$ be a radius such that
  $\overline{B_R(z_0)} \subset V$. For a sequence of radii
  $0 < r_\eta \leq R$ and a sequence of open sets
  $I_\eta \subset (0,1)$, we consider $T_{r_\eta, I_\eta}$ of
  \eqref{def:T_I,r}.  The inclusion
  $\Sigma_Y^\eta = T_{r_\eta, I_\eta}$ is always asymptotically
  disconnected in direction $e_2$, i.e.~there exists a sequence
  $\Phi_\eta \in \Hd(\curl, Z)$ satisfying \eqref{eq:CellProblem:Phi}
  for $j = 2$. In fact, one can choose $\Phi_\eta= \Phi$ independent
  of $\eta$.
\end{proposition}

\begin{proof}
  We use $\phi$ of Lemma~\ref{lem:phi-wires} for $U = V_R$ and
  $\Phi_\eta = \Phi$ by \eqref{eq:Con:Phi-phi}. Since $\phi$ is
  periodic in the first argument, $\Phi$ is $e_2$-periodic (since
  $\Phi$ is independent of $x_1$ it is also $e_1$-periodic). The curl
  of $\Phi$ is essentially given by $\nabla^\perp \cdot \phi$, see
  \eqref{eq:curl-Phi=div-phi}. Using $\nabla^\perp \cdot \phi = 0$ of
  \eqref{eq:phi:curl}, we obtain $\curl \Phi = 0$. This implies
  \eqref{eq:CellProblem:Phi:curl}.

  The estimate \eqref{eq:CellProblem:Phi:L2} compares the cell
  function $\Phi$ with the constant $e_2$ on the set $Z$. Since $\Phi$
  is identical to $e_2$ for $x_3 \not \in (0,1)$, the only
  contributions to the $L^2(Z)$-norm come from $Y$, where the
  $L^2$-norm of $\Phi$ is finite, hence
  \begin{align*}
    \eta^{1/2}\|\Phi - e_2\|_{L^2(Z)}
    &= \eta^{1/2} \| \Phi - e_2 \|_{L^2(Y)}
      \to 0 \,.
  \end{align*}
  This shows that $\Phi_\eta = \Phi$ satisfies
  \eqref{eq:CellProblem:Phi:L2}.
  Moreover, since $\phi$ vanishes on $B_R(z_0) = V \setminus U$, see
  \eqref{eq:phi:innen}, we have $\Phi|_{T_{R, \emptyset}} =0$ and,
  thus, in particular, $\Phi|_{T_{r_\eta, I_\eta}}=0$, which verifies
  \eqref{eq:CellProblem:Phi:obstacle}.
\end{proof}

Proposition \eqref{prop:construction:Phi-wires-ortho} and its proof
hold also for more general geometries:

\begin{remark}[Disconnectedness of general cylinders in
  direction~$e_2$]\label{rem:compact-inclusions:2D}
  Let all sets $\Sigma_Y^\eta$ be contained in a set
  $\Sigma_Y^0= (0,1) \times B$ for a closed set
  $B \subset V = (0,1)^2$ that is independent of $\eta$. Then,
  $\Sigma^\eta_Y$ is asymptotically disconnected in directions $e_2$.
\end{remark}

\begin{proof}
  We note that $V \setminus B$ has one connected component
  $U\subset V\setminus B$ for which $\partial U \supset \partial V$
  holds.  Arguing as in the proof of
  Proposition~\ref{prop:construction:Phi-wires-ortho} using the set
  $U$ instead of $V_R$, we obtain that
  $(0,1) \times (V \setminus \overline{U})$ is asymptotically
  disconnected in the direction $e_2$.  For $\Sigma_Y^0$ and
  $\Sigma_Y^\eta$, the asymptotically disconnectedness in direction
  $e_2$ follows from the monotonicity result of
  Proposition~\ref{prop:GeometricMonotonicity}.
\end{proof}

\begin{proof}[Proof of Remark \ref{rem:compact-inclusions}]
  For the closed set $\Sigma_Y^0\subset Y = (0,1)^3$ one can find a
  closed set $B_1\in (0,1)^2$ such that
  $\Sigma_Y^0 \subset (0,1) \times B_1$.  Choosing the larger set
  $(0,1) \times B_1$ as new $\Sigma_Y^0$,
  Remark~\ref{rem:compact-inclusions:2D} shows the asymptotic
  disconnectedness in direction $e_2$.  Moreover, there exists a set
  $B_2$ such that
  $\Sigma_Y^0 \subset \{y \in Y \mid (y_1,y_3) \in B_2 \}$. By
  changing the role of $x_1$ and $x_2$, we obtain the asymptotic
  disconnectedness in direction $e_1$.
\end{proof}

\section{Criticality of the radius asymptotics}
\label{sec:criticality}
In this section we show that, when the radii are substantially smaller
than in Proposition \ref{prop:construction:Psi-wires-perfect}, the
cylindrical obstacles are disconnected even in direction $e_1$. This is
true even if there are no gaps.

\begin{lemma}[2D scalar field $\phi_r$]\label{lem:phi-wires-critical}
  Let $z_0 \in V = (0,1)^2$ be a point, $R>0$ be a radius such that
  $\overline{B_R(z_0)} \subset V$. We consider $0<r\leq R^2$ and
  $V_r = V \setminus B_r(z_0)$. There exists $\phi_{r} \in H^1(V,\R)$
  and a constant $C_4$, which does not depend on $r$, such that
  \begin{align}\label{eq:boundary-cond-phi-r-M-critical}
    \phi_{r}|_{B_r(z_0)}
    &\equiv 0\, , \qquad
      \phi_{r}|_{V \setminus B_R(z_0)} \equiv 1 \,,
    \\
    \label{eq:est:phi_r-critical}
    \phi_r(x) & \in [0,1] \qquad \textrm{for a.e. } x \in V\,,
    \\\label{eq:est:phi_r:curl-critical}\
    \| \nabla \phi_{r} \|^2_{L^2(V)}
    &\leq C_4 |\ln(r)|^{-1} \,.
  \end{align}
\end{lemma}

\begin{proof}
  We define $\phi_r$ on $B_R(z_0) \setminus B_r(z_0)$ as the minimizer of
  the Dirichlet energy
  $\frac{1}{2}\|\nabla \phi_r\|_{L^2(B_R(z_0) \setminus B_r(z_0))}^2$
  subject to the boundary values $\phi_r = 0$ on $\partial B_r(z_0)$ and
  $\phi_r = 1$ on $\partial B_R(z_0)$.  This minimizer is given by a
  logarithm in the radius, more precisely, with
  $|x| = (x_1^2 +x_2^2)^{1/2}$:
  \begin{equation*}
    \phi_r(x) = \widetilde{\phi}_r(|x-z_0|) \qquad \text{for } \quad 
    \widetilde{\phi}_r(\tau) = \frac{\ln(\tau) - \ln(r)}{\ln(R) - \ln(r)}.
  \end{equation*}
  We extend $\phi_r$ by $0$ in $B_r(z_0)$ and by $1$ in
  $V \setminus B_R(z_0)$ such that
  \eqref{eq:boundary-cond-phi-r-M-critical} and \eqref
  {eq:est:phi_r-critical} are satisfied.  Elementary calculus yields
  \begin{align*}
    \|\nabla \phi_{r} \|^2_{L^2(V)}
    &= \int_{B_R(z_0) \setminus B_r(z_0)} |\nabla \phi_r|^2
    = 2 \pi \int_r^R |\del_\tau \widetilde{\phi}_r|^2\, \tau\, d\tau\\
    &=2 \pi(\ln(R) - \ln(r))^{-2}\int_r^R \frac1{\tau}\, d\tau
      =2 \pi(\ln(R) - \ln(r))^{-1}\,.
  \end{align*}
  This provides \eqref{eq:est:phi_r:curl-critical}.
\end{proof}

As in Section \ref {sec:wires}, we construct a sequence of $3$D cell
functions from the $2$D counterpart. Here, for a sequence of radii
$r_\eta$ we use the $2$D scalar field $\phi_{r_\eta}$ from
Lemma~\ref{lem:phi-wires-critical} to construct a field $\Phi_\eta$
that is aligned with $e_1$:
\begin{align}
  \label{eq:Con:Phi-phi-critical}
  &\Phi_\eta(x_1,x_2,x_3) \coloneqq e_1 \, \phi_{r_\eta}(x_2,x_3)\, \one{\{0< x_3<1\}}
    + e_1\, \one{\{x_3>1\text{ or } x_3<0\}} \,.
\end{align}
The properties of $\phi_{r_\eta}$ on $\{z_2 = 1\}$ and
$\{z_2 = 0\}$ imply that the tangential components of the $3$D
function $\Phi_\eta$ have no jumps.  The $\curl$ is therefore given by
\begin{align}
  \label{eq:curl-Phi=div-phi-critical}
  &\curl \Phi_\eta(x_1,x_2,x_3)
    = 
    \left(\begin{array}{c}
            0 \\
            \del_2 \phi_{r_\eta}\\
            -\del_1 \phi_{r_\eta} 
          \end{array}
  \right) (x_2,x_3)\ \one{\{0< x_3<1\}}\,.
\end{align}

\begin{proposition}[Cell function $\Phi_\eta$ and disconnectedness in
  direction $e_1$ for too thin wires]
  \label{prop:construction:Phi-wires-thin}
  Let $z_0 \in V = (0,1)^2$ be a point and $R>0$ with
  $\overline{B_R(z_0)} \subset V$. For a sequence of radii $0 < r_\eta \leq R^2$,
  we consider $T_{r_\eta}$ defined in \eqref{def:T_r}. Let the
  inclusions be wires $\Sigma_Y^\eta = T_{r_\eta}$ with exponentially
  small radius in the sense that
  \begin{equation}
    \label{eq:critical-8232}
    \eta |\ln(r_\eta)| \to \infty \, .
  \end{equation} 
  Then, $\Sigma^\eta_Y$ is asymptotically disconnected in direction
  $e_1$ in the sense of Definition~\ref
  {def:Analytical-DisConnectedness}: There exists a sequence
  $\Phi_\eta \in \Hd(\curl, Z)$ satisfying \eqref{eq:CellProblem:Phi}
  with $j = 1$.
\end{proposition}

\begin{proof}
  We consider $\phi_{r_\eta} \in H^1(V)$ of Lemma \ref{lem:phi-wires}
  and use $\Phi_\eta$ of \eqref {eq:Con:Phi-phi-critical}.  The
  function $\phi_{r_\eta}$ is periodic in its first argument,
  therefore, $\Phi_\eta$ is $e_2$-periodic (it is independent of
  $x_1$, hence also $e_1$-periodic).  The curl of $\Phi_\eta$ is given
  by \eqref{eq:curl-Phi=div-phi-critical} and we have
  $\Phi_\eta \in \Hd(\curl, Z)$.  It remains to show
  \eqref{eq:CellProblem:Phi}.
  
  By construction, $\phi_{r_\eta}|_{B_{r_\eta}(z_0)} = 0$. This yields that
  $\Phi_\eta$ vanishes in $T_{r_\eta}$, thus,
  \eqref{eq:CellProblem:Phi:obstacle} holds.  Regarding
  \eqref{eq:CellProblem:Phi:L2}, we calculate
  \begin{align*}
    \eta^{1/2} \|\Phi_\eta - e_1\|_{L^2(Z)}
    &= \eta^{1/2} \|e_1 (\phi_{r_\eta} - 1)\|_{L^2(V)}
      \leq \eta^{1/2} \to 0 \,,
  \end{align*}
  where we used $\phi_{r_\eta}(x) \in [0,1]$, which was observed in
  \eqref{eq:est:phi_r-critical}.  We have thus obtained
  \eqref{eq:CellProblem:Phi:L2}.
  
  In order to show \eqref{eq:CellProblem:Phi:curl}, we use the curl of
  $\Phi_\eta$ that was calculated in
  \eqref{eq:curl-Phi=div-phi-critical}.  The estimate
  \eqref{eq:est:phi_r:curl-critical} on $\nabla \phi_{r_\eta}$ allows to
  calculate
  \begin{align*}
    \eta^{-1}\|\curl \Phi_\eta \|_{L^2(Z)}^2
    &= \eta^{-1} \|\nabla^\perp \phi_{r_\eta} \|_{L^2(V)}^2
      \leq \eta^{-1} C_4 |\ln(r_\eta)|^{-1}\to 0\,,
  \end{align*}
  where we used \eqref {eq:critical-8232} in the convergence.  This
  verifies \eqref{eq:CellProblem:Phi:curl}.
\end{proof}

\def\cprime{$'$}

%\bibliography{lit-all-lsi}
\end{document}